\pdfoutput=1
\RequirePackage{ifpdf}
\ifpdf 
\documentclass[pdftex]{sigma}
\else
\documentclass{sigma}
\fi

\numberwithin{equation}{section}

\newtheorem{Theorem}{Theorem}[section]

\newtheorem{Lemma}[Theorem]{Lemma}
{ \theoremstyle{definition}
\newtheorem{Definition}[Theorem]{Definition}

\newtheorem{Example}[Theorem]{Example}
\newtheorem{Remark}[Theorem]{Remark} }

\begin{document}
\allowdisplaybreaks

\newcommand{\arXivNumber}{1812.04511}

\renewcommand{\PaperNumber}{056}

\FirstPageHeading

\ShortArticleName{Invariant Nijenhuis Tensors and Integrable Geodesic Flows}

\ArticleName{Invariant Nijenhuis Tensors\\ and Integrable Geodesic Flows}

\Author{Konrad LOMPERT~$^\dag$ and Andriy PANASYUK~$^\ddag$}

\AuthorNameForHeading{K.~Lompert and A.~Panasyuk}

\Address{$^\dag$~Faculty of Mathematics and Information Science, Warsaw University of Technology,\\
\hphantom{$^\dag$}~ul.~Koszykowa~75, 00-662 Warszawa, Poland}
\EmailD{\href{mailto:k.lompert@mini.pw.edu.pl}{k.lompert@mini.pw.edu.pl}}

\Address{$^\ddag$~Faculty of Mathematics and Computer Science, University of Warmia and Mazury,\\
\hphantom{$^\ddag$}~ul.~S{\l}oneczna 54, 10-710 Olsztyn, Poland}
\EmailD{\href{mailto:panas@matman.uwm.edu.pl}{panas@matman.uwm.edu.pl}}

\ArticleDates{Received December 19, 2018, in final form August 02, 2019; Published online August 07, 2019}

\Abstract{We study invariant Nijenhuis $(1,1)$-tensors on a homogeneous space $G/K$ of a reductive Lie group $G$ from the point of view of integrability of a Hamiltonian system of differential equations with the $G$-invariant Hamiltonian function on the cotangent bundle~$T^*(G/K)$. Such a tensor induces an invariant Poisson tensor $\Pi_1$ on $T^*(G/K)$, which is Poisson compatible with the canonical Poisson tensor~$\Pi_{T^*(G/K)}$. This Poisson pair can be reduced to the space of $G$-invariant functions on $T^*(G/K)$ and produces a family of Poisson commuting $G$-invariant functions. We give, in Lie algebraic terms, necessary and sufficient conditions of the completeness of this family. As an application we prove Liouville integrability in the class of analytic integrals polynomial in momenta of the geodesic flow on two series of homogeneous spaces $G/K$ of compact Lie groups $G$ for two kinds of metrics: the normal metric and new classes of metrics related to decomposition of $G$ to two subgroups $G=G_1\cdot G_2$, where $G/G_i$ are symmetric spaces, $K=G_1\cap G_2$.}


\Classification{37J15; 37J35; 53D25}

\section{Introduction}

By Maupertuis's principle integrability of the geodesic flow of a (pseudo-)Riemannian metric is a question as old as classical mechanics itself. In this paper we consider Hamiltonian systems and understand integrability in the sense of Arnold--Liouville, i.e., as existence of a complete family of first integrals in involution. The Clairaut theorem on existence of linear integral for the motion of a free particle on a surface of revolution is traditionally mentioned as one of the first results on Arnold--Liouville integrability of geodesic flows. Next classical cases are the Euler top and geodesics on ellipsoid. In modern mathematical literature one could find many examples of integrable geodesic flows on homogeneous spaces of Lie groups starting probably with the papers~\cite{mishch,thimm}, see also the review~\cite{bjObzor} and references therein and later works \cite{DGJ,jovanovic10,mykSuborb,p4}.

The present paper continues this line and develops a new approach for constructing integrable geodesic flows on homogeneous spaces. Let $G$ be a reductive Lie group, $K \subset G$ its closed subgroup. The cotangent bundle $T^*(G/K)$ with its canonical Poisson structure $\Pi$ is a phase space of a Hamiltonian system with the Hamiltonian function equal to the quadratic form~$q$ of an $G$-invariant pseudo-Riemannian metric, which can be constructed as follows. Let $\langle\,,\,\rangle$ be an $\operatorname{Ad} G$-invariant symmetric bilinear form on~$\mathfrak{g}$, the Lie algebra of $G$. It gives rise to a bi-invariant metric on $G$, which induces on $G/K$ an $G$-invariant metric $\langle\,,\,\rangle_{G/K}$ called \emph{normal}. Besides, one can consider a symmetric $\operatorname{ad} \mathfrak{k}$-invariant linear operator (called \emph{inertia operator}) $n_{\mathfrak{k}^\perp}\colon \mathfrak{k}^\perp\to\mathfrak{k}^\perp$, where $\mathfrak{k}$ is the Lie algebra of $K$ and $\mathfrak{k}^\perp$ is its orthogonal complement in $\mathfrak{g}$ with respect to~$\langle\,,\,\rangle$. It will give rise to an $G$-invariant $(1,1)$-tensor $N\colon T(G/K)\to T(G/K)$, which is symmetric with respect to $\langle\,,\,\rangle_{G/K}$, and to another $G$-invariant metric $\langle\cdot,\cdot\rangle_N:=\langle N\cdot,\cdot\rangle_{G/K}$. The question of integrability of the geodesic flows of both metrics $\langle\,,\,\rangle_{G/K}$ and $\langle\cdot,\cdot\rangle_N$ on $G/K$ consists of finding a family of $\dim(G/K)-1$ independent analytic and \emph{polynomial in momenta} functions on $T^*(G/K)$ which Poisson commute with the quadratic form $q$ and with each other. It is known \cite[Section~5]{bjObzor} that there are two families $\mathcal{F}_1$ and~$\mathcal{F}_2$ of analytic functions on $T^*(G/K)$ that Poisson commute with each other, in which one can look for desirable integrals. These are the family $\mathcal{F}_2$ of $G$-invariant functions and the family $\mathcal{F}_1$ of functions of the form~$\mu_{\rm can}^*f$, where $\mu_{\rm can}\colon T^*(G/K)\to\mathfrak{g}^*$ is the momentum map corresponding to the natural Hamiltonian action of~$G$ on $T^*(G/K)$ and $f$ is an analytic function on $\mathfrak{g}^*$. Obviously $q\in\mathcal{F}_2$ and taking a family $\mathcal{F}$ of commuting polynomials on $\mathfrak{g}^*$ (by the Sadetov theorem~\cite{sadetov} there exist \emph{complete} such families, see also~\cite{bolsinov16}) one gets the family $\mathcal{A}:=\mu_{\rm can}^*(\mathcal{F})$ of integrals of $q$ polynomial in momenta. Thus the problem now is reduced to the following one: construct a family $\mathcal {B}\subset\mathcal{F}_2$ of commuting polynomial in momenta integrals of $q$ such that the family $\mathcal{A}+\mathcal {B}$ is complete.

An approach for constructing such a family $\mathcal{B}$ was proposed in~\cite{p4}. The homogeneous spaces considered were the coadjoint orbits $\mathcal{O}$ of $G$. A second $G$-invariant Poisson structure $\Pi_1$ was constructed on $T^*(G/K)$ which is compatible with $\Pi$ and the family $\mathcal{B}$ was the canonical family of functions in involution related with the Poisson pair $(\Pi',\Pi_1')$ being the reduction of the Poisson pair $(\Pi,\Pi_1)$ with respect to the action of $G$. Essential role in the construction of $\Pi_1$ played the Kirillov--Kostant--Suriau symplectic form $\omega_\mathcal{O}$ on $\mathcal{O}$, as $\Pi_1=(\omega+\pi^*\omega_\mathcal{O})^{-1}$, where $\omega=-\Pi^{-1}$ is the canonical symplectic form on $T^*\mathcal{O}$ and $\pi\colon T^*\mathcal{O}\to \mathcal{O}$ is the canonical projection.

In this paper we propose a novel approach for constructing the family $\mathcal{B}$. Similarly to the case above, we construct a second Poisson structure $\Pi_1$ compatible with $\Pi$, but we use invariant Nijenhuis $(1,1)$-tensors $N\colon T(G/K)\to T(G/K)$ for this purpose instead, in particular avoiding the restriction on $G/K$ of being a coadjoint orbit. In more detail, $\Pi_1=\widetilde{N}\circ\Pi$, where $\widetilde{N}$ is the so-called \emph{cotangent lift} of $N$, see Definition~\ref{ctgl}. Obviously, an invariant $(1,1)$-tensor on~$G/K$ is determined by a linear operator $n\colon \mathfrak{g}\to\mathfrak{g}$. We get some Lie algebraic conditions on this operator which are necessary and sufficient for the so-called \emph{kroneckerity} of the Poisson pair~$(\Pi',\Pi_1')$ obtained as the reduction of the pair $(\Pi,\Pi_1)$ and, as a consequence, of the completeness of the family $\mathcal{B}$ (and $\mathcal{A}+\mathcal {B}$), see Theorem~\ref{kro}, the main result of this paper, and Theorem~\ref{thgeod}. As an application we construct two series of invariant Nijenhuis $(1,1)$-tensors on homogeneous spaces~$G_k/K_k$ of compact simple Lie groups, where $(G_k,K_k)$ is $({\rm SU}(2k),{\rm S}({\rm U}(2k-1)\times {\rm U}(1))\cap {\rm Sp}(k))$ or $({\rm SO}(2k+2),{\rm SO}(2k+1)\cap {\rm U}(k+1))$, which lead to invariant metrics with geodesic flow Liouville integrable in the class of integrals analytic and polynomial in momenta (Theorem~\ref{appl}). Besides we prove integrability of the normal metric on these homogeneous spaces. Below the content of the paper is discussed in more detail.

In Section~\ref{s.1} we study Lie algebraic conditions on the operator $n\colon \mathfrak{g}\to\mathfrak{g}$ which guarantee the vanishing of the Nijenhuis torsion of~$N$ (Theorem~\ref{T1}) and consider some examples.

A crucial role in our considerations play bi-Hamiltonian (bi-Poisson) structures, i.e., pencils of Poisson structures generated by pairs of compatible ones. We devote Section \ref{sec:Bi-Poisson0} to related notions and preparatory results which will enable us to study the completeness of families of functions in involution. Theorem~\ref{below} gives some criteria of completeness of the canonical family of $G$-invariant functions related to an action of a Lie group~$G$ on a bi-Poisson manifold~$M$ being Hamiltonian with respect to almost all Poisson structures from the pencil. The theorem requires some assumptions among which the most significant one says that the action of $G$ on $M$ is \emph{locally free}. This assumption enables to use the so-called bifurcation lemma and to prove the constancy of rank of the reduced bi-Poisson structure for almost all values of the parameter, which is a~first step for achieving the kroneckerity.

In Section~\ref{sec:Bi-Poisson} we study bi-Poisson structures on $T^*(G/K)$ generated by Poisson pairs $(\Pi,\Pi_1=\widetilde{N}\circ\Pi)$, where $N$ is a semisimple invariant Nijenhuis $(1,1)$-tensor. We show that almost all (\emph{generic}) Poisson structures from the corresponding Poisson pencil are nondegenerate and calculate the dimensions of the symplectic leaves of the\emph{ exceptional} (not being generic) Poisson structures (Lemma~\ref{L56}). We prove the hamiltonicity of canonical action of~$G$ on $T^*(G/K)$ with respect to the generic Poisson structures, as well as the hamiltonicity of the actions of some subgroups (stabilizers of the symplectic leaves) on the symplectic leaves of the exceptional ones. We calculate the corresponding momentum maps (see Lemma~\ref{lemk}) as well as these stabilizers (Lemma~\ref{stab}).

The main result, Theorem \ref{kro}, which gives necessary and sufficient conditions for kroneckerity of the reduced Poisson pair~$(\Pi',\Pi_1')$ in terms of the \emph{indices} of the Lie algebra $\mathfrak{g}$ and some its contractions (see formula~(\ref{eq78})), is proved in Section~\ref{s.kro}. As a corollary we prove Theorem~\ref{thgeod} stating the complete integrability of the geodesic flow of the normal metric and the metric with the inertia operator~$n|_{\mathfrak{k}^\perp}$ under the assumption that the sufficient conditions from Theorem~\ref{kro} are satisfied.

In Section~\ref{s.appl} we apply the above results to construct examples of metrics with integrable geodesic flow. The main idea which enables to fit conditions of Theorem~\ref{kro} is based on the Brailov theorem (see Theorem \ref{brailov}) stating equality of indices of a semisimple Lie algebra and its $\mathbb{Z}_2$-\emph{contractions}. We observe that among the examples of invariant Nijenhuis $(1,1)$-tensors on a homogeneous space $G/K$ from Section~\ref{s.1} related to the Onishchik list of decompositions $\mathfrak{g}=\mathfrak{g}_1+\mathfrak{g}_2$ of a simple compact Lie algebra to two subalgebras (Example~\ref{onish}) there are two series $(\mathfrak{g}(k),\mathfrak{g}_1(k),\mathfrak{g}_2(k))$ in which both the pairs $(\mathfrak{g}(k),\mathfrak{g}_1(k))$ and $(\mathfrak{g}(k),\mathfrak{g}_2(k))$ are \emph{symmetric}, i.e., by the Brailov theorem these examples satisfy conditions~(\ref{eq78}) of Theorem~\ref{kro} (the Lie algebra $\mathfrak{k}$ of the group $K$ is equal $\mathfrak{g}_1(k)\cap\mathfrak{g}_2(k)$). In order to apply this theorem for the proof of complete integrability of the geodesic flow one needs only ensure that the action of $G$ on $T^*(G/K)$ is locally free. This is done in the proof of Theorem~\ref{appl} stating the complete integrability of the geodesic flows of the normal metric and the metric with the corresponding inertia operator.

The explicit formulae for the realizations of Lie algebras $\mathfrak{g}(k)$, $\mathfrak{g}_1(k)$, $\mathfrak{g}_2(k)$ for both series as well as for the corresponding inertia operators are given in Appendix~\ref{appe}. There we also indicate conditions under which these operators (and the corresponding metrics) are positive definite. We end the paper by concluding remarks (Section~\ref{concl}) in which we discuss some details of the paper and possible perspectives.

Fix some notations. We write $P\colon G\to G/K$, $\pi\colon T^*M\to M$, and $p\colon M\to M/G$ for the canonical projections.

All objects in this paper are real analytic or complex analytic. Given a vector bundle $E$, we write $\Gamma(E)$ for the space of sections of $E$, and $\mathcal{E}(M)$ will stand for the space of functions on a~mani\-fold $M$ (of the corresponding category).

\section{Invariant Nijenhuis tensors on homogeneous spaces}\label{s.1}

\begin{Definition}\label{nijdef}Let $M$ be a connected manifold. A $(1,1)$-tensor field $N\colon TM\to TM$ is \emph{a~Nijenhuis tensor} if its Nijenhuis torsion vanishes,
i.e., for any vector fields $X,Y\in \Gamma(TM)$:
\begin{gather*}
 T_N(X,Y):=[NX,NY]-N [X,Y]_N=0,
\end{gather*}
where we put
\begin{gather*}
[X,Y]_N:=[NX,Y]+[X,NY]-N[X,Y].
\end{gather*}
Similarly, given any Lie algebra $(\mathfrak{g}, [\,,\,])$, a linear operator $n\colon \mathfrak{g}\to\mathfrak{g}$ is an \emph{algebraic Nijenhuis operator} if it satisfies
$T_n(X,Y):=[nX,nY]-n ( [nX,Y] + [X,nY] - n[X,Y] )=0$ for all vectors $X,Y\in \mathfrak{g}$ (cf.~\cite{golsok2000,mks}).
\end{Definition}

Let an action of a Lie group $G$ on a manifold $M$ be given.

\begin{Definition}\label{def2}We say that a $(1,1)$-tensor field $N\colon TM\to TM$ is $G$-\emph{invariant} if for any element of the Lie group $ g\in G$, the tensor $N$ commutes with the tangent map $g_*\colon TM \to TM$ to the diffeomorphism $g\colon M\to M$, i.e., the following diagram is commutative
\begin{gather*}
\begin{matrix}
TM & \stackrel{N}{\longrightarrow}& TM \\
\downarrow g_*& & \downarrow g_* \\
TM &\stackrel{N}{\longrightarrow}&TM .
\end{matrix}
\end{gather*}
A distribution of subspaces $D_x \subset T_xM$ is $G$-\emph{invariant}, if for any $g\in G$ and any $x\in M$ we have
\begin{gather*}
g_{*,x}(D_x)=D_{gx}.
\end{gather*}
\end{Definition}

The following lemmas are crucial ingredients in further considerations. Let $G$ be any Lie group and $K$ a closed Lie subgroup of $G$ (the quotient $M=G/K$ is then a smooth $G$-manifold).

\begin{Lemma}\label{L1}Let $N\colon T(G/K) \to T(G/K)$ be a semisimple $(1,1)$-tensor and assume that $N$ is $G$-invariant. Then the eigenvalues of $N$ are constant.
\end{Lemma}
\begin{proof}Since the operator $N$ is $G$-invariant, it follows that its eigenfunctions are also $G$-invariant, therefore on homogeneous space they are constant functions.
\end{proof}

Given a real manifold $M$, we write $T^\mathbb{C} M$ for the complexified tangent bundle to $M$.

\begin{Lemma}\label{L3}Let $G/K$ be a real homogeneous space. There is a one-to-one correspondence between $G$-invariant distributions $D \subset T^\mathbb{C}(G/K)$ and subspaces $ d \subset\mathfrak{g}^\mathbb{C}$ such that $\mathfrak{k}^\mathbb{C} \subset d$ and $\big[\mathfrak{k}^\mathbb{C}, d\big]\subset d$ $($here $\mathfrak{g},\mathfrak{k} \subset\mathfrak{g}$ are the Lie algebras of the Lie groups $G, K \subset G)$. An $G$-invariant distribution $D$ is involutive if and only if the subspace $ d$ is a subalgebra in $\mathfrak{g}^\mathbb{C}$. Moreover, $D$ is real, i.e., $D=\overline D$, where the bar stands for the complex conjugation on $T^\mathbb{C}(G/K)$, if and only if so is $ d$, i.e., $d=\bar d$, where the bar denotes the complex conjugation in $\mathfrak{g}^\mathbb{C}$ with respect to the real form $\mathfrak{g}$.
\end{Lemma}

\begin{proof} Below we let $P\colon G\to G/K$ to denote the canonical projection. An invariant distribution~$D$ on $G/K$ defines the distribution $\widehat D:=P_*^{-1}(D)\subset T^\mathbb{C} G$, which by construction is left $G$-invariant. Indeed, the invariance of~$D$, $g_{*,x}(D_x)=D_{gx}$, implies $L_g$-invariance of $\widehat D$ as the commutativity of the following diagram shows
\begin{gather*}
\begin{matrix}
T_yG & \stackrel{L_{g,*}|_y}{\longrightarrow}& T_{gy}G \\
\downarrow P_{*,y}& & \downarrow P_{*,gy} \\
T_x(G/K) &\stackrel{g_{*,x}}{\longrightarrow}&T_{gx}(G/K);
\end{matrix}
\end{gather*}
 here $L_g$ is the left translation by $g$ and $y\in G$ is so that $P(y)=x$.

Moreover, $\widehat D$ is right $K$-invariant. To show this observe that, since $P$ is a surjective submersion, in a vicinity of points $g\in G$ and $P(g)\in G/K$ there exist local coordinate systems $(x_1,\dots,x_m, y_1,\dots,y_k)$ and $(x'_1,\dots,x'_m)$ respectively such that $P(x_1,\dots,x_m,y_1,\dots,y_k)=(x'_1,\dots,x'_m)$, $x'_i=x_i$, $i=1,\dots,m$. Let $X_1(x'),\dots,X_l(x')$, $X_r(x')=X_r^i(x_1',\dots,x_m') \frac{\partial}{\partial x_i'}$, be local linearly independent vector fields on $G/K$ generating the distribution~$D$. Then the distribution~$\widehat D$ is generated by the vector fields $\widehat X_r(x)=X_r^i(x_1,\dots,x_m) \frac{\partial }{\partial x_i}$, $r=1,\dots,l$, and $Y_1,\dots,Y_k$, where the last ones are the fundamental vector fields of the right $K$-action. These last are tangent to the fibers of $P$, locally can be expressed as combinations of $\frac{\partial }{\partial y_j}$ and vice versa, $\frac{\partial }{\partial y_j}$~can be locally expressed as combinations of $Y_1,\dots,Y_k$. Obviously, $\big[Y_i,\widehat X_j\big]=f_{ij}^sY_s$ for some functions~$f_{ij}^s$, which together with the involutivity of the system of vector fields $\{Y_1,\dots,Y_k\}$ gives $\big[Y_i,\widehat{D}\big] \subset\widehat{D}$.

Let $ d:= \widehat D_e \subset T^\mathbb{C}_eG$, where $e\in G$ is the neutral element. The left and right $K$-invariance of~$\widehat D$ implies, under the identification $T^\mathbb{C}_eG\cong\mathfrak{g}^\mathbb{C}$, the $\operatorname{Ad} (K)$-invariance of the subspace $ d \subset \mathfrak{g}^\mathbb{C}$, or, on the infinitesimal level, its $\operatorname{ad}(\mathfrak{k})$-invariance: $\big[\mathfrak{k}^\mathbb{C}, d \big] \subset d$.

Now if $D$ is involutive, then so is $\widehat D$. Indeed, the systems of vector fields $\{X_j\}$ and, consequently, $\big\{\widehat X_j\big\}$ are involutive. Hence so is the total system of vector fields $\big\{Y_i,\widehat X_j\big\}$. Infinitesimally this can be expressed as $[ d, d] \subset d$.

Vice versa, let $ d \subset\mathfrak{g}^\mathbb{C}\cong T^\mathbb{C}_eG$ be an $\operatorname{ad}\mathfrak{k}$-invariant subspace. Define a distribution $\widehat D \subset T^\mathbb{C} G$ by $\widehat D_g=L_{g,*} d$. Then $\widehat D$ is left $G$-invariant and right $K$-invariant and descends to a uniquely defined invariant distribution $D \subset T^\mathbb{C}(G/K)$ by means of the complexified tangent map $P_*^\mathbb{C}\colon \allowbreak T^\mathbb{C} G\to T^\mathbb{C}(G/K)$.

If $ d \subset \mathfrak{g}^\mathbb{C}$ is a subalgebra, then clearly the distribution $\widehat D$ is involutive. Moreover, from the above local description it follows that the system of vector fields $\big\{ \widehat X_j \big\} $ is involutive and $P_* \widehat X_j = X_j $, and, as a consequence, so is the system~$\{X_j\}$. Therefore $ D\subset T^\mathbb{C} (G/K)$ is involutive.

The last assertion of the lemma is obvious.
\end{proof}

\begin{Lemma}\label{L5}Let $D \subset T(G/K)$ be an $G$-invariant integrable distribution on $G/K$ relative to a~subalgebra $\mathfrak{h}\subset \mathfrak{g}$, $\mathfrak{h}\supset \mathfrak{k}$ $($as in Lemma~{\rm \ref{L3}} but we admit also the complex analytic case$)$, and let $H \subset G$ be the corresponding subgroup. Denote by $P\colon G\to G/K$ the canonical projection. Then
\begin{enumerate}\itemsep=0pt
\item[$1)$] the leaves of the foliation tangent to $D$ are the projections with respect to $P$ of the left cosets $gH$, $g\in G$;
\item[$2)$] given $\xi\in\mathfrak{g}$, the fundamental vector field $X_\xi$ of the $G$-action on $G/K$ is tangent to the leaf $P(gH)$ if and only if $\xi \in \operatorname{Ad}_g \mathfrak{h} \subset \mathfrak{g} $.
 \end{enumerate}
\end{Lemma}

\begin{proof} Consider the integrable distribution $\widehat D$ built in the proof of Lemma~\ref{L3}. Then it is easy to see that the foliation tangent to~$\widehat D$ coincides with the foliation of the left cosets $gH$, $g\in G$. Since $P_*\big(\widehat D\big)=D$, the leaves of the corresponding foliations are projected on each other by means of~$P$, which proves item~1.

The right invariant vector field $\xi_R $ on $G$, $\xi_R|_e=\xi$, is
 tangent to $gH$ at the point $gh\in gH$ if and only if
$\xi_R ( gh ) \in T_{gh} ( gH) \Leftrightarrow R_{gh,*}(\xi) \in T_{gh} (gH) = L_{gh,*} \mathfrak{h} \Leftrightarrow \xi \in R_{(gh)^{-1},*} L_{gh,*} \mathfrak{h} = \operatorname{Ad}_{gh} \mathfrak {h}=\operatorname{Ad}_g\mathfrak{h} $ (here $\xi\in\mathfrak{g}\cong T_eG$). Hence $X_\xi=P_*\xi_R$ is tangent to $P(gH)$ if and only if $\xi \in \operatorname{Ad}_g \mathfrak{h}$.
\end{proof}

\begin{Lemma}\label{L4} Let $N\colon TM\to TM$ be a semisimple $(1,1)$-tensor with constant distinct eigenvalues $\lambda_1,\dots,\lambda_s \in\mathbb{C}$ $($or $\lambda_1,\dots,\lambda_s \in\mathbb{R})$ and let $D_i \subset T^\mathbb{C} M$ $($or, respectively $D_i \subset T M)$ be the eigendistribution corresponding to~$\lambda_i$. Then $T_N=0$ if and only if the distributions $D_i$ and $D_i+D_j$ are involutive for any~$i$,~$j$.
\end{Lemma}

\begin{proof} Assume $N$ is Nijenhuis. It is easy to see that $T_{N-\lambda I}=T_N=0$ for any $\lambda \in \mathbb{C} $. In particular $\big[\big(N^\mathbb{C}-\lambda_i I\big)X,\big(N^\mathbb{C}-\lambda_i I\big)Y\big]=\big(N^\mathbb{C}-\lambda_i I\big)[X,Y]_{N^\mathbb{C}-\lambda_i I}$ for any vector fields~$X$,~$Y$ and the image of $N^\mathbb{C}-\lambda_i I\colon T^\mathbb{C} M\to T^\mathbb{C} M$ is an integrable distribution. As a consequence, $D_i=\bigcap_{k\not=i}\operatorname{im}\big(N^\mathbb{C}-\lambda_k I\big)$ and $D_i+D_j=\bigcap_{k\not=i,j}\operatorname{im}\big(N^\mathbb{C}-\lambda_k I\big)$ are integrable.

Now, let the decomposition $T^\mathbb{C} M= D_1\oplus \cdots \oplus D_s$ be such that $D_{i}+D_{j}$ are integrable for any~$i$,~$j$. By the bilinearity of Nijenhuis torsion tensor it is enough to prove that $T_N (x,y)=0$ for $x\in \Gamma(D_i)$, $y\in \Gamma(D_j)$, $1\le i,j\le n$:
\begin{gather*}
T_N(x,y)= [Nx,Ny]-N([Nx,y]+[x,Ny])+N^2[x,y] \\
\hphantom{T_N(x,y)}{} =\lambda_i\lambda_j([x,y]_i+[x,y]_j)-N(\lambda_i ([x,y]_i+[x,y]_j)\\
\hphantom{T_N(x,y)=}{} + \lambda_j([x,y]_i+[x,y]_j))+N(\lambda_i[x,y]_i+\lambda_j[x,y]_j) \\
\hphantom{T_N(x,y)}{} = \lambda_i\lambda_j([x,y]_i+[x,y]_j)-\big(\lambda_i^2[x,y]_i
+\lambda_i\lambda_j[x,y]_j+\lambda_i\lambda_j[x,y]_i+\lambda_j^2[x,y]_j\big)\\
\hphantom{T_N(x,y)=}{} +(\lambda_i^2[x,y]_i+\lambda_j^2[x,y]_j) = 0
\end{gather*}
(here we denote by $[x,y]_i$ the $i$-th component of the element $[x,y]$ with respect to the decomposition above). The proof in the case of real eigenvalues is analogous.
\end{proof}

Let $G$ be any Lie group and $\mathfrak{g}=\operatorname{Lie}(G)$ its Lie algebra, $K$ a closed Lie subgroup of $G$ and $\mathfrak{k}=\operatorname{Lie}(K)$.

\begin{Theorem}\label{T1}There is a one-to-one correspondence between
\begin{itemize}\itemsep=0pt
\item[$(i)$] $ G-$invariant semisimple Nijenhuis $(1,1)$-tensors $N\colon T(G/K) \to T(G/K) $
with the spectrum
$\{ \lambda_1, \dots, \lambda_s \} $, where $\lambda_i$ are distinct, $\lambda_1,\dots,\lambda_{2p} \in \mathbb{C}$, $\lambda_i=\overline{\lambda_{i+p}}$ for $i=1,\dots,p$ and $\lambda_{2p+1},\dots,\lambda_s\in\mathbb{R}$,
\end{itemize}
and
\begin{itemize}\itemsep=0pt
\item[$(ii)$] decompositions $\mathfrak{g}^\mathbb{C} = \mathfrak{g}_1 + \dots+ \mathfrak{g}_s $ of $\mathfrak{g}^\mathbb{C}$ to the sum of subspaces such that:
\begin{enumerate}\itemsep=0pt
\item[$1)$] $ \forall_{i,j\in \{1,\dots,s\}, i\not=j }$ $\mathfrak{g}_i\cap\mathfrak{g}_j = \mathfrak{k}^\mathbb{C}$;
\item[$2)$] the induced decomposition of the factor space $\mathfrak{g}^\mathbb{C}/\mathfrak{k}^\mathbb{C}$ is direct: $\mathfrak{g}^\mathbb{C}/\mathfrak{k}^\mathbb{C}=\big(\mathfrak{g}_1/\mathfrak{k}^\mathbb{C}\big) \oplus \cdots \oplus\big(\mathfrak{g}_s/\mathfrak{k}^\mathbb{C}\big)$;
\item[$3)$] $ \forall_{i,j\in \{1,\dots,s\} }$ $\mathfrak{g}_i+\mathfrak{g}_j $ are Lie subalgebras in $\mathfrak{g}^\mathbb{C}$;
\item[$4)$] $\mathfrak{g}_i=\overline{\mathfrak{g}_{i+p}}$ for $i=1,\dots,p$ and $\mathfrak{g}_j=\overline{\mathfrak{g}_j}$ for $j=2p+1,\dots,s$.
\end{enumerate}
\end{itemize}
The decomposition $(ii)$ induces the decomposition $T^\mathbb{C}(G/K)=D_1\oplus\cdots\oplus D_s$ to involutive subbundles, the corresponding $(1,1)$-tensor $N$ is then given by $N|_{D_i}=\lambda_i\operatorname{Id}_{D_i}$ and, vice versa, given~$N$ as in~$(i)$ one constructs the decomposition $(ii)$ by the decomposition $T^\mathbb{C}(G/K)=D_1\oplus\cdots\oplus D_s$ of $T^\mathbb{C}(G/K)$ to the eigendistributions of~$N$.
\end{Theorem}

\begin{proof}Let $N$ be an $ G$-invariant semisimple Nijenhuis $(1,1)$-tensor on $G/K$ with the spectrum $\{ \lambda_1, \dots, \lambda_s ; \, \lambda_i \in \mathbb{C} , \, \lambda_i \neq \lambda_j ,\text{ for } i\neq j \} $. From Lemma~\ref{L4} it follows that there is a decomposition $ T^\mathbb{C} (G/K) = D_1 \oplus\cdots \oplus D_s $ into integrable distributions, which, as the eigenspaces of an $G$-invariant tensor, are also $G$-invariant. By Lemma~\ref{L3} there is a one-to-one correspondence between $G$-invariant distributions~$D_i$ and subalgebras $\mathfrak{g}_i$ containing $\mathfrak{k}^\mathbb{C}$, hence there is a decomposition of $\mathfrak{g}^\mathbb{C}=\mathfrak{g}_1 + \cdots +\mathfrak{g}_s $, such that $\mathfrak{g}_i\cap\mathfrak{g}_j = \mathfrak{k}^\mathbb{C} $ for any $i\not=j$. Applying Lemma~\ref{L3} to the sum of distributions $D_i + D_j$ we see that it is involutive if and only if $\mathfrak{g}_i+\mathfrak{g}_j$ is a subalgebra.

Item 3 follows from the last assertion of Lemma \ref{L3} and from the obvious fact that $D_i=\overline{D_{i+p}}$ for $i=1,\dots,p$ and $D_j=\overline{D_j}$ for $j=2p+1,\dots,s$.

The proof in reverse direction follows the same argumentation with the use of the equivalences in lemmas cited.
\end{proof}

Below we present some examples for which decompositions of Lie algebras mentioned in Theorem~\ref{T1} are given explicitly.

First series of examples come from semisimple algebraic Nijenhuis operators $n\colon \mathfrak{g} \to \mathfrak{g} $, which are $\operatorname{ad} \mathfrak{k} $-invariant for some Lie subalgebra $\mathfrak{k}\subset \mathfrak{g}$, i.e., $ n \circ \operatorname{ad} k = \operatorname{ad} k \circ n $ for all $ {k\in \mathfrak{k}}$. Then by $\operatorname{ad} \mathfrak{k}$-invariance we can extend it to an invariant Nijenhuis $(1,1)$-tensor $N$ on~$G/K$.

In the literature the following two classes of algebraic Nijenhuis operators are widely known~\cite{golsok2000,mks,pBi-Lie}:\footnote{The is one more class defined on the full matrix algebra by $nX=AXB+BAX$, where $A^2=B^2=I$ \cite{odsokIntMatrEq}. For some particular cases of the matrices $A$ and $B$ the corresponding operator is semisimple. However these cases are beyond the scope of this paper.} first is related to a direct decomposition of the algebra~$\mathfrak{g}$ to two subalgebras, second is related to the operator of left multiplication on the full matrix algebra. Below we consider particular cases of these two classes.

\begin{Example}\label{ex1} Let $\mathfrak{g}$ be a semisimple Lie algebra with the root system $R$ with respect to a~Cartan subalgebra $\mathfrak{h} \subset\mathfrak{g}$. Let $\mathfrak{g}=\mathfrak{h}+ \sum\limits_{\alpha\in R}\mathfrak{g}_\alpha$ be the corresponding root decomposition. Choose $R^+$ and $R^-$ to be sets of positive and negative roots and let $S\subset \Pi$ be any subset of the set of positive simple roots. We denote by~$[S]$ the set of positive roots generated by~$S$. Consider the decomposition $\mathfrak{g} = \mathfrak{p} \oplus \mathfrak{p}^{\perp} $, where $\mathfrak{p}:=\mathfrak{h}+\sum\limits_{\alpha \in R^-}\mathfrak{g}_\alpha+\sum\limits_{\alpha \in [S]}\mathfrak{g}_\alpha$ is the corresponding parabolic subalgebra and $\mathfrak{p} ^{\perp}=\sum\limits_{\alpha \in R^+\setminus [S]}\mathfrak{g}_\alpha$ (the orthogonal complement with respect to Killing form). Then $\mathfrak{p} ^{\perp}$ is obviously a subalgebra too. The operator $n\colon \mathfrak{g}\to\mathfrak{g}$ defined by $ n|_{\mathfrak{g}_1 } = \lambda_1 \operatorname{Id}_{\mathfrak{g}_1}$, $n|_{\mathfrak{g}_2 } = \lambda_2 \operatorname{Id}|_{\mathfrak{g}_2}$ with $\mathfrak{g}_1=\mathfrak{p}$ and $\mathfrak{g}_2=\mathfrak{p}^\bot$ and with arbitrary $\lambda_1$, $\lambda_2 $ is algebraic Nijenhuis (cf.~\cite{golsok2000,p5}).

One may take $\mathfrak{k} = \mathfrak{p} \cap \mathfrak{p}^{\text{opposite}} $, where $ \mathfrak{p}^{\text{opposite}}= \mathfrak{h}+ \sum\limits_{\alpha \in -[S]\subset R^-}\mathfrak{g}_\alpha+\sum\limits_{\alpha \in R^+}\mathfrak{g}_\alpha$. Then the operator $n$ will be $\operatorname{ad}\mathfrak{k}$-invariant and will generate an $G$-invariant Nijenhuis $(1,1)$-tensor on~$G/K$, where $G$, $K \subset G$ are the corresponding Lie groups. The decomposition of Theorem \ref{T1} looks as follows: $\mathfrak{g}_1:=\mathfrak{p}, \mathfrak{g}_2:=\mathfrak{p}^{\text{opposite}}$.
An instance of such a situation for $\mathfrak{g}=\mathfrak{sl}(3,\mathbb{R})$ can be schematically presented as
\begin{gather*}
\mathfrak{p}=\begin{bmatrix}*&0&0\\ *&*&* \\ *&*&* \end{bmatrix}, \qquad
\mathfrak{p}^\perp = \begin{bmatrix}0&*&*\\ 0&0&0 \\ 0&0&0\end{bmatrix}, \qquad
\mathfrak{p}^{\text{opposite}} = \begin{bmatrix}*&*&*\\ 0&*&*\\ 0&*&*\end{bmatrix}, \qquad
\mathfrak{k} = \begin{bmatrix}*&0&0\\ 0&*&*\\ 0&*&*\end{bmatrix},
\end{gather*}
where the corresponding set $S$ consists of the sole root $e_2-e_3$, $e_i(H)$ being the $i$-th diagonal element of $H\in\mathfrak{h}$.
\end{Example}

\begin{Example}\label{1b}Let $\mathfrak{g}=\mathfrak{gl}(n,\mathbb{K})$, $\mathbb{K}=\mathbb{R},\mathbb{C}$, and consider $n=L_A$, the operator of left multiplication by a matrix $A\in\mathfrak{g}$. Then it is easy to see that $n$ is an algebraic Nijenhuis operator. Taking $A=\operatorname{diag}(\lambda_1, \dots , \lambda_n)$, $\lambda_i\not=\lambda_j$, $i\not=j$, we get a semisimple operator, whose eigenspaces $\ker(n-\lambda_i\operatorname{Id})$ consist of matrices with the only nonzero $i$-th row. Obviously, $n$ is $\operatorname{ad} \mathfrak{k}$-invariant for $\mathfrak{k}= Z(A)$, the centralizer of $A$, which coincides with the
 subalgebra of diagonal matrices. The decomposition of Theorem~\ref{T1} is $\mathfrak{g}=\sum\limits_{i=1}^n\mathfrak{g}_i$, where $\mathfrak{g}_i=\ker(n-\lambda_i\operatorname{Id})+\mathfrak{k}$ consists of the matrices having non zero elements at most on the diagonal and $i$-th row.

The generalization to the case when multiplicities in the spectrum of $A$ are admitted is straightforward. This example has also an obvious generalization to the case $\mathfrak{g}=\mathfrak{sl}(n,\mathbb{K})$.
\end{Example}

Our next example is quite classical, as this is the complex structure operator on the adjoint orbits of the compact Lie groups which was intensively studied in the literature. We adapt the description of this operator to our notations. An alternative description can be found in \cite[Chapter~8.B]{besse}.

\begin{Example}\label{exa2}Let $\mathfrak{g}$ be a complex semisimple Lie algebra, $\mathfrak{h} \subset\mathfrak{g}$ a Cartan subalgebra, $\mathfrak{g}=\mathfrak{h}+\sum\limits_{\alpha\in R}\mathfrak{g}_\alpha$ the corresponding root grading. For any $\alpha\in R$ choose $E_\alpha\in\mathfrak{g}_\alpha$ such that $\langle E_\alpha,E_{-\alpha}\rangle=1$ and put $H_\alpha:=[E_\alpha,E_{-\alpha}]$. Then $\mathfrak{u}=\sum\limits_{\alpha\in R^+}\mathbb{R}({\rm i} H_\alpha)+\sum\limits_{\alpha\in R^+}\mathbb{R}(E_\alpha-E_{-\alpha})+\sum\limits_{\alpha\in R^+}\mathbb{R}({\rm i}(E_\alpha+E_{-\alpha})) \subset\mathfrak{g}$, where $R^+ \subset R$ is a subset of positive roots, is the compact real form of $\mathfrak{g}$ \cite[Theorem~6.3, Chapter~III]{helgason}. By \cite[Theorem~1.3, Chapter~6]{vinbergOnishIII} the centralizer $Z_\mathfrak{u}(a)$ of any element $a\in\mathfrak{u}$ (which is necessarily semisimple) is of the form $Z_\mathfrak{u}(a)=\sum\limits_{\alpha\in R^+}\mathbb{R}({\rm i} H_\alpha)+\sum\limits_{\alpha\in [S]}\mathbb{R}(E_\alpha-E_{-\alpha})+\sum\limits_{\alpha\in [S]}\mathbb{R}({\rm i}(E_\alpha+E_{-\alpha}))$, where $S \subset R^+$ is a subset of the set of simple positive roots (cf.\ Example~\ref{ex1}). Consider the operator $j\colon \mathfrak{u}^\perp\to \mathfrak{u}^\perp$, where $\mathfrak{u}^\perp:=\sum\limits_{\alpha\in R^+\setminus[S]}\mathbb{R}(E_\alpha-E_{-\alpha})+\sum\limits_{\alpha\in R^+\setminus[S]}\mathbb{R}({\rm i}(E_\alpha+E_{-\alpha}))$, given by $ j(E_\alpha - E_{-\alpha} ) = {\rm i} ( E_\alpha + E_{-\alpha})$, $j( {\rm i} (E_\alpha + E_{-\alpha} ))=- ( E_\alpha - E_{-\alpha})$. Note that $j^\mathbb{C}(E_\alpha)={\rm i}E_\alpha$, $j^\mathbb{C}(E_{-\alpha})=-{\rm i}E_{-\alpha}$. The eigenspaces $\mathfrak{g}_1':=\sum\limits_{\alpha\in R^+\setminus[S]}\mathbb{C}(E_\alpha)$ and $\mathfrak{g}_2':=\sum\limits_{\alpha\in R^+\setminus[S]}\mathbb{C}(E_{-\alpha})$ are subalgebras as well as the subspaces $\mathfrak{g}_i:=\mathfrak{g}_i'\oplus\mathfrak{k}^\mathbb{C}$, $\mathfrak{k}=Z_\mathfrak{u}(a)$. Hence by Theorem~\ref{T1} the operator $j$ induces an invariant integrable almost complex structure on~$U/K$, where $U,K \subset U$ are the Lie groups corresponding to the Lie algebras $\mathfrak{u}$, $\mathfrak{k}$. We conclude that, although this operator is not arising from an algebraic Nijenhuis operator, the corresponding decomposition in fact coincides with that from Example~\ref{ex1}).
\end{Example}

Now we come to a series of examples of different nature.\footnote{By this we mean that they are not necessarily related with an $\operatorname{ad}\mathfrak{k}$-invariant algebraic Nijenhuis operator on the Lie algebra $\mathfrak{g}$. For instance, in Example~\ref{exa4} there are two ways to build a compatible with the decomposition $\mathfrak{g}=\mathfrak{g}_1+\mathfrak{g}_2$ algebraic Nijenhuis operator $n\colon \mathfrak{g}\to\mathfrak{g}$: $n|_{\mathfrak{g}_1}=\lambda_1\operatorname{Id}_{\mathfrak{g}_1}$, $n|_{\mathfrak{g}_2'}=\lambda_2\operatorname{Id}_{\mathfrak{g}_2'}$ with $\mathfrak{g}_2'=\mathfrak{g}_1^\bot$, or $n|_{\mathfrak{g}_2}=\beta_2\operatorname{Id}_{\mathfrak{g}_2}$, $n|_{\mathfrak{g}_1'}=\beta_1\operatorname{Id}_{\mathfrak{g}_1'}$ with $\mathfrak{g}_1'=\mathfrak{g}_2^\bot$. However, in both the cases the operator in general will not be $\operatorname{ad}\mathfrak{k}$-invariant. Concerning the decompositions from the Onishchik list, see Example \ref{onish}, it seems that it is even impossible to build a compatible Nijenhuis operator for some of them.} The decomposition of Theorem~\ref{T1} will still consist of two components which now need not be symmetric with respect to the involution interchanging $\mathfrak{g}_\alpha$ and $\mathfrak{g}_{-\alpha}$. In other words, any decomposition $\mathfrak{g}=\mathfrak{g}_1+\mathfrak{g}_2$ of a Lie algebra $\mathfrak{g}$ to two subalgebras can be taken into consideration (with $\mathfrak{k}=\mathfrak{g}_1\cap\mathfrak{g}_2$). One of possible natural generalizations of Example \ref{ex1} is considering two ``nonsymmetric'' parabolic subalgebras. Their intersection is the so-called \emph{seaweed subalgebra}.

\begin{Example}\label{exa4}Let $\mathfrak{g}$ be a semisimple Lie algebra with the root system $R$ with respect to a~Cartan subalgebra $\mathfrak{h} \subset\mathfrak{g}$. Let $\mathfrak{g}=\mathfrak{h}+ \sum\limits_{\alpha\in R}\mathfrak{g}_\alpha$ be the corresponding root decomposition. Choose~$R^+$ and~$R^-$ to be sets of positive and negative roots and let $S,S'\subset \Pi$ be any subsets of the set of positive simple roots. Consider the parabolic subalgebras $\mathfrak{g}_1 = \mathfrak{h}+\sum\limits_{\alpha \in R^-}\mathfrak{g}_\alpha+\sum\limits_{\alpha \in [S]}\mathfrak{g}_\alpha$ and $\mathfrak{g}_2 = \mathfrak{h}+\sum\limits_{\alpha \in R^+}\mathfrak{g}_\alpha+\sum\limits_{\alpha \in -[S']}\mathfrak{g}_\alpha$. An instance of such a situation for $\mathfrak{g}=\mathfrak{sl}(3,\mathbb{R})$ can be schematically presented as
\begin{gather*}
\mathfrak{g}_1=\begin{bmatrix}*&0&0\\ *&*&* \\ *&*&* \end{bmatrix}, \qquad
\mathfrak{g}_2 = \begin{bmatrix}*&*&*\\ *&*&*\\ 0&0&*\end{bmatrix}, \qquad
\mathfrak{k} = \begin{bmatrix}*&0&0\\ *&*&*\\ 0&0&*\end{bmatrix},
\end{gather*}
where the corresponding sets $S$ and $S'$ consist of the roots $e_2-e_3$ and $e_1-e_2$ respectively, cf.\ Example~\ref{ex1}.
\end{Example}

In \cite{oni62} A.L.~Onishchik classified all decompositions $\mathfrak{g}=\mathfrak{g}_1 +\mathfrak{g}_2$ for compact simple Lie algebras~$\mathfrak{g}$ and we list them below. (In \cite{oni69} he also gave a classification of decompositions of reductive Lie algebras $\mathfrak{g}$ to two subalgebras reductive in $\mathfrak{g}$, but we omit this case here.)

\begin{Example}\label{onish} Let $\mathfrak{g}$ be a compact simple Lie algebra. The following table presents all pairs of subalgebras $(\mathfrak{g}_1,\mathfrak{g}_2)$ such that $\mathfrak{g}=\mathfrak{g}_1+\mathfrak{g}_2$ together with possible embeddings $i'\colon \mathfrak{g}_1\rightarrow\mathfrak{g}$, $i''\colon \mathfrak{g}_2\rightarrow\mathfrak{g}$ up to conjugations. Below~$N$ stands for the trivial representation, $\varphi_i$ for the specific representation mentioned in~\cite{oni62} and $T$ for the 1-dimensional Lie algebra.

\begin{table}[t]\centering
\begin{tabular}{|c|c|c|c|c|c|c|}
\hline
$\mathfrak{g}$ & $\mathfrak{g}_1$ & $i'$ & $\mathfrak{g}_2$ & $i''$ & $\mathfrak{k}=\mathfrak{g}_1\cap \mathfrak{g}_2$ & restrictions \\
\hline \hline
$A_{2n-1}$ & $C_n$ & $\varphi _1$ & $A_{2n-2}$ & $\varphi _1+N$ & $C_{n-1}$ & $n>1$\\
\hline
$A_{2n-1}$ & $C_n$ & $\varphi _1$ & $A_{2n-2}\oplus T$ & $\varphi _1+N$ & $C_{n-1}\oplus T$ & $n>1$\\
\hline
$B_3$ & $G_2$ & $\varphi _2$ & $B_2$ & $\varphi _1+2N$ & $A_{1}$ & \null \\
\hline
$B_3$ & $G_2$ & $\varphi _2$ & $B_2\oplus T$ & $\varphi _1+2N$ & $A_{1}\oplus T$ & \null \\
\hline
$B_3$ & $G_2$ & $\varphi _2$ & $D_3$ & $\varphi _1+N$ & $A_{2}$ & \null \\
\hline
$D_{n+1}$ & $B_n$ & $\varphi _1+N$ & $A_n$ & $\varphi _1+\varphi _n$ & $A_{n-1}$ & $n>2$\\
\hline
$D_{n+1}$ & $B_n$ & $\varphi _1+N$ & $A_n\oplus T$ & $\varphi _1+\varphi _n$ & $A_{n-1}\oplus T$ & $n>2$ \\
\hline
$D_{2n}$ & $B_{2n-1}$ & $\varphi _1+N$ & $C_n$ & $\varphi _1+\varphi _1$ & $C_{n-1}$ & $n>1$\\
\hline
$D_{2n}$ & $B_{2n-1}$ & $\varphi _1+N$ & $C_n\oplus T$ & $\varphi _1+\varphi _1$ & $C_{n-1}\oplus T$ & $n>1$ \\
\hline
$D_{2n}$ & $B_{2n-1}$ & $\varphi _1+N$ & $C_n\oplus A_1$ & $\varphi _1+\varphi _1$ & $C_{n-1}\oplus A_{1}$ & $n>1$ \\
\hline
$D_8$ & $B_7$ & $\varphi _1+N$ & $B_4$ & $\varphi _4$ & $B_{3}$ & \null \\
\hline
$D_{4}$ & $B_3$ & $\varphi _3$ & $B_2$ & $\varphi _1+3N$ & $A_{1}$ & \null \\
\hline
$D_{4}$ & $B_3$ & $\varphi _3$ & $B_2\oplus T$ & $\varphi _1+3N$ & $A_{1}\oplus T$ & \null\\
\hline
$D_{4}$ & $B_3$ & $\varphi _3$ & $B_2\oplus A_1$ & $\varphi _1+3N$ & $A_{1}\oplus A_{1}$ & \null \\
\hline
$D_{4}$ & $B_3$ & $\varphi _3$ & $D_3$ & $\varphi _1+2N$ & $A_{2}$ & \null \\
\hline
$D_{4}$ & $B_3$ & $\varphi _3$ & $D_3\oplus T$ & $\varphi _1+2N$ & $A_{2}\oplus T$ & \null \\
\hline
$D_{4}$ & $B_3$ & $\varphi _3$ & $B_3$ & $\varphi _1+N$ & $G_{2}$ & \null \\
\hline
\end{tabular}
\end{table}
\end{Example}

\section[Bi-Poisson structures, kroneckerity, $G$-invariance, and complete families of functions in involution]{Bi-Poisson structures, kroneckerity, $\boldsymbol{G}$-invariance,\\ and complete families of functions in involution}\label{sec:Bi-Poisson0}

If $M$ is a real or complex analytic manifold, ${\mathcal E}(M)$ will stand for the space of analytic functions on~$M$ in the corresponding category. We will write $\mathbb{K}$ for the corresponding ground field. We recall basic definitions and concepts related to bi-Poisson structures, their kroneckerity and invariance (cf.~\cite{p4}).

We will say that some functions from the set ${\mathcal E}(M)$ are independent at a point $x\in M$ if their differentials are independent at~$x$. For any subset ${\mathcal F}\subset{\mathcal E}(M)$ denote by $\operatorname{ddim}_x {\mathcal F}$ the maximal number of independent functions from the set ${\mathcal F}$ at a~point $x\in M$. Put $\operatorname{ddim} {\mathcal F}:= \max\limits_{x\in M}\operatorname{ddim}_x{\mathcal F}$.

\begin{Definition}\label{10.70}A bivector field (bivector for short) is a skew-symmetric morphism $\Pi\colon T^*M\to TM$. It is called \emph{Poisson} if the operation $\{f,g\}^\Pi:=\Pi(f)g$ is a Lie algebra on $\mathcal{E}(M)$ (here $\Pi(f):=\Pi(df)$ stands for the \emph{Hamiltonian vector field} corresponding to the function $f$). Define $\operatorname{rank} \Pi:=\max\limits_{x\in M}\dim\Pi(T_x^*M)$ and $R^\Pi:=\{x\in M\,|\, \dim\Pi(T_x^*M)=\operatorname{rank} \Pi\}$.
A function $f\in \mathcal{E}(U)$ over an open set $U\subset M$ is called a \emph{Casimir function} for a Poisson bivector $\Pi$ if $\Pi(f)\equiv 0$. The set of all Casimir functions for $\Pi$ over $U$ will be denoted by $Z^\Pi(U)$ (note that $Z^\Pi(U)$ is the centre of the Lie algebra $\big(\mathcal{E}(U),\{\,,\,\}^\Pi\big)$).
\end{Definition}

Given a poisson bivector $\Pi$, the generalized distribution $\Pi(T^*M)\subset TM$ (called the \emph{characteristic distribution} of~$\Pi$) is integrable, the restrictions of $\Pi$ to its leaves are correctly defined nondegenerate Poisson bivectors and the leaves are called the \emph{symplectic leaves} of~$\Pi$. In particular the set $R^\Pi$ is the union of all the symplectic leaves of maximal dimension.

\begin{Definition}\label{10.80}A set $I\subset \mathcal{E}(U)$ of functions over $U\subset M$ is called involutive with respect to a Poisson bivector $\Pi$ if $\{f,g\}^\Pi=0$ for any $f,g\in I$ (we also say that such functions are \emph{in involution}). An involutive set is \emph{complete} with respect to $\Pi$ if there exist $f_1,\dots,f_s\in I$, where $s=\dim M-\frac{1}{2}\operatorname{rank} \Pi$, independent at any point from some open dense set $U_0 \subset U$.
\end{Definition}

If $I$ is a complete involutive set over $U$, then among $f_i$ there are $\dim M-\operatorname{rank} \Pi$ Casimir functions of~$\Pi$. Any such set $I$ is a set of functions constant along a lagrangian foliation of dimension $\frac12\operatorname{rank} \Pi$ defined on an open dense set in any symplectic leaf of maximal dimension.

\begin{Definition}\label{de.1.1}Two Poisson structures $\Pi_1$ and $\Pi_2$ on a manifold $M$ is called \emph{compatible} if $\Pi_t := t_1\Pi_1+t_2\Pi_2$ is a Poisson bivector for any $t=(t_1,t_2)\in \mathbb{K}^2$; the whole 2-dimensional family of Poisson bivectors (in case $\Pi_{1,2}$ are linearly independent) $\{\Pi_t\}_{t\in\mathbb{R}^2}$ is called a {\it bi-Poisson} or {\it a~bi-Hamiltonian structure}.
\end{Definition}

\begin{Definition}\label{de.1.2}A bi-Poisson structure $\{\Pi_t\}$ on $M$ is {\it Kronecker at a point} $x\in M$ if $\operatorname{rank}_\mathbb{C} (t_1\Pi_1+t_2\Pi_2)|_x$ is constant with respect to $(t_1,t_2)\in\mathbb{C}^2\setminus\{0\}$ (in the real analytic case we consider $(\Pi_j)_x$ as a skew-symmetric bilinear form on the complexified cotangent space $(T^*_xM)^\mathbb{C}$). We say that~$\{\Pi_t\}$ is \emph{Kronecker} if it is Kronecker at any point of some open dense subset in $M$.
\end{Definition}

Importance of this notion is explained by the following

\begin{Theorem}\label{prop}Let $\{\Pi_t\}$ be a Kronecker bi-Poisson structure on $M$. Then for any open set $U \subset M$ such that $\operatorname{ddim} Z^{\Pi_t}(U)=\dim M-\operatorname{rank} \Pi_t$ for any $t$ the set
\begin{gather*}
Z^{\{\Pi_t\}}(U):=\operatorname{Span}\bigg(\bigcup_{t\not=0}Z^{\Pi_t}(U)\bigg)
\end{gather*}
is a complete involutive set of functions with respect to any $\Pi_t\not =0$ $($see Definition~{\rm \ref{10.80})}.
\end{Theorem}
The reader is referred to~\cite{bols'} for the proof. The condition that $\operatorname{ddim} Z^{\Pi_t}(U)=\dim M-\operatorname{rank} \Pi_t$ for any $t$ is always satisfied for any sufficiently small open set~$U$ and, in many cases also for an open and dense set in~$M$.

\begin{Remark}\label{rem}Recall that a real analytic submanifold~$M$, $\dim_\mathbb{R} M=n$, in a complex manifold~${M}^c$, $\dim_\mathbb{C}{M}^c=n$, is called \emph{maximal totally real} if in a neighbourhood of any point in $M$ there exists a holomorphic coordinate system $z=(z_1,\dots,z_n)$, $z_j=x_j+{\rm i}y_j$, such that~$M$ locally is given by the equations $y_j=0$. We say that ${M}^c$ is a \emph{complexification} of~$M$ and~$M$ is \emph{real form} of~${M}^c$. The holomorphic coordinates as above will be called \emph{adapted} to $M$. A complexification exists for any real analytic~$M$~\cite{bw}.

Let $M$ be a real analytic manifold and ${M}^c$ its complexification. Any real analytic tensor~$T$ defined on~$M$ can be uniquely extended to a~holomorphic tensor ${T}^c$ defined in a vicinity of~$M$ in~${M}^c$ by extending its coefficients to holomorphic functions and substituting $\frac{\partial }{\partial x_j}$ and ${\rm d}x_j$ by~$\frac{\partial }{\partial z_j}$ and~${\rm d}z_j$ respectively in the adapted systems of coordinates. Vice versa, if a holomorphic tensor~${T}^c$ is given on~${M}^c$ such that in the adapted coordinates its coefficients restricted to~$M$ are real, then it is the holomorphic extension of some real analytic tensor~$T$ on~$M$. Obviously, if $\{\Pi_t\}_{t\in\mathbb{R}^2}$, $\Pi_t=t_1\Pi_1+t_2\Pi_2$, is a real analytic bi-Poisson structure on~$M$, then it is Kronecker at a point $m\in M$ if and only if so is its holomorphic extension $\{{\Pi}^c\}_{t\in\mathbb{C}^2}$, ${\Pi}^c_t=t_1{\Pi}^c_1+t_2 {\Pi}^c_2$.

\end{Remark}Let $G$ be a Lie group acting on a manifold $M$. Denote by $\mathcal{E}^G(M)$ the space of all $G$-invariant functions from the set ${\mathcal E}(M)$. We say that a bi-Poisson structure $\{\Pi_t\}$ is $G$-\emph{invariant} if so is each bi-vector $\Pi_t$, $t\in\mathbb{R}^2$.

Now we assume that the action of $G$ on $M$ is proper, as for instance is in the case of any smooth action of a compact Lie group. Fix some isotropy subgroup $H\subset G$ determining the {\em principal orbit type}. In this case the subset
\begin{gather*}
M_{H}=\big\{x\in M\colon G_x=gHg^{-1}\ \mathrm{for\ some}\ g\in G\big\}
\end{gather*}
of $M$, consisting of all orbits $G\cdot x$ in~$M$ isomorphic to~$G/H$, is an open and dense subset of $M$ (see~\cite[Section~2.8 and Theorem~2.8.5]{DK}). It is well known that the orbit space $M'_H:=M_{H}/G$ is a smooth manifold. There is a natural identification of spaces $\mathcal{E}^G(M_{H})$ and $p^*\mathcal{E}(M'_{H})$, where $p\colon M_{H}\to M'_H=M_{H}/G$ is the canonical projection, in particular $\operatorname{ddim}_x \mathcal{E}^G(M)=\operatorname{ddim} \mathcal{E}^G(M)$ for $x\in M_{H}$. Moreover, if an $G$-invariant bi-Hamiltonian structure $\{\Pi_t\}$ is given on $M$, all the Poisson bivectors $\Pi_t|_{M_{H}}$ are projectable with respect to~$p$, i.e., there exist a correctly defined bi-Poisson structure $\{\Pi'_t\}$ on $M_{H}/G$ such that $\Pi'_t=p_*\Pi_t$, and the identification mentioned is a~Poisson map:
\begin{gather*}
p^*\{f,g\}^{\Pi'_t}=\{p^*f,p^*g\}^{\Pi_t}, \qquad f,g\in\mathcal{E}(M'_H).
\end{gather*}
Assuming that the reduced bi-Poisson structure $\{\Pi'_t\}$ is Kronecker, by Theorem \ref{prop} for a sufficiently small $U \subset M'_H$ we get an involutive family of functions
\begin{gather*}
Z^{\{\Pi'_t\}}(U):=\operatorname{Span}\bigg(\bigcup_{t\not=0}Z^{\Pi'_t}(U)\bigg),
\end{gather*}
which is complete with respect to any Poisson structure $\Pi'_t$. In some cases the corresponding set of functions $p^*Z^{\{\Pi'_t\}}(U)$ on $p^{-1}(U) \subset M$ which by the considerations above is involutive with respect to any Poisson bivector $\Pi_t$ can be extended to a complete involutive set of functions. One such situation is touched in Theorem~\ref{below} below. This theorem also describes a method of proving the kroneckerity of the bi-Poisson structure $\{\Pi'_t\}$ reducing the problem to the calculation of rank of a finite number of the reduced Poisson structures, which was used in \cite{p4,p3}.

\begin{Theorem}\label{below} Retaining the assumptions above assume moreover that
\begin{enumerate}\itemsep=0pt
\item[$(a)$] the associated action $\rho\colon \mathfrak{g}\to\Gamma(TM)$ of the Lie algebra $\mathfrak{g}$ of $G$ on $M$ can be extended to a~holomorphic action $\rho^c\colon \mathfrak{g}^\mathbb{C}\to\Gamma(T {M}^c)$ of the complexification $\mathfrak{g}^\mathbb{C}$ of the Lie algebra $\mathfrak{g}$ on some complexification ${M}^c$ of $M$ on which a holomorphic extension $\{{\Pi}^c_t\}$ of $\{\Pi_t\}$ is defined $($see Remark~{\rm \ref{rem})} and $\{{\Pi}^c_t\}$ is $\mathfrak{g}^\mathbb{C}$-invariant, i.e., the Lie derivative $\mathcal{L}_{\rho^c(\xi)} {\Pi}^c_t$ is equal to zero for any $t\in\mathbb{C}^2$ and any $\xi\in\mathfrak{g}^\mathbb{C}$; here $\Gamma(TM)$ stands for the space of real analytic vector fields on $M$ and $\Gamma(T {M}^c)$ for the space of holomorphic vector fields on ${M}^c$;
\item[$(b)$] the action of $G$ on $M$ is generically {\it locally free}, i.e., the stabilizer $H$ corresponding to the
principal orbit type is finite; in particular, a generic stabilizer algebra of the actions~$\rho$ and~$\rho^c$ is trivial;
\item[$(c)$] $\operatorname{codim}\operatorname{Sing}\mathfrak{g}^*\ge 2$, where $\operatorname{Sing}\mathfrak{g}^* \subset \mathfrak{g}^*$ is the union of the coadjoint orbits of nonmaximal dimension, i.e., $\operatorname{Sing}\mathfrak{g}^*=\mathfrak{g}^*\setminus R^{\Pi_{\mathfrak{g}^*}}$ for the Lie--Poisson structure $\Pi_{\mathfrak{g}^*}$ on $\mathfrak{g}^*$;
\item[$(d)$] for almost all $t$ the bivector ${\Pi}^c_t$ is nondegenerate and the action $\rho^c$ is Hamiltonian with respect to ${\Pi}^c_t$, i.e., there exists a set~$E \subset\mathbb{C}^2$ being the union of a finite number of $1$-dimensional linear subspaces $\langle t_1\rangle, \dots,\langle t_s\rangle$, a map $\mu^c_t\colon {M}^c\to \big(\mathfrak{g}^\mathbb{C}\big)^*$, $t\in\mathbb{C}^2\setminus E$ $($the so-called momentum map$)$, such that $\operatorname{rank} {\Pi}^c_t=\dim M$, $t\in \mathbb{C}^2\setminus E$, and any fundamental vector field $\rho^c(\xi)$, $\xi\in\mathfrak{g}^\mathbb{C}$, of this action is a Hamiltonian vector field ${\Pi}^c_t\big(H_t^\xi\big)$ with the Hamiltonian function $H_t^\xi(x)=\langle \mu^c_t(x),\xi\rangle$ and $\mu^c_t$ is a Poisson map from the Poisson manifold $({M}^c,{\Pi}^c_t)$ to the Lie--Poisson manifold $\big(\big(\mathfrak{g}^\mathbb{C}\big)^*,\Pi_{(\mathfrak{g}^\mathbb{C})^*}\big)$;
 \item[$(e)$] the restriction $\mu_t=\mu^c_t|_M$, $t\in\mathbb{R}^2$, takes values in $\mathfrak{g}^* \subset \big(\mathfrak{g}^\mathbb{C}\big)^*$; in particular the action $\rho$ itself is Hamiltonian with respect to any $\Pi_t$, $t\not\in\mathbb{R}^2\cap E$: $\rho(\xi)=\Pi_t\big(H_t^\xi|_M\big)$, $\xi\in\mathfrak{g}$.
\end{enumerate}
Then
\begin{enumerate}\itemsep=0pt
\item[$1)$] the set $U:=M_H\setminus\big(\bigcup_{t\in\mathbb{R}^2}\mu_t^{-1}(\operatorname{Sing}\mathfrak{g}^*)\big)$ is an $G$-invariant open dense set in~$M_H$;
\item[$2)$] the reduced bi-Poisson structure $\{\Pi'_t\}$ on $M'_H=M_H/G$ is Kronecker at a~point $x'\in p(U)$ if and only if
\begin{gather*}
\operatorname{corank} \Pi'_{t_i}|_{x'}=\operatorname{ind} \mathfrak{g}, \qquad i=1,\dots,s;
\end{gather*}
\item[$3)$] if $\{\Pi'_t\}$ is Kronecker and $\mathcal{F}$ stands for any complete involutive set of polynomial functions on~$(\mathfrak{g}^*,\Pi_{\mathfrak{g}^*})$ $($which exists by the Sadetov theorem~{\rm \cite{sadetov})}, the set of functions
 \begin{gather}\label{setoffu}
 \mathcal{I}:= p^*\big(Z^{\{\Pi'_t\}}(M_H')\big)\bigcup \mu_{t_0}^*\mathcal{F}
 \end{gather}
is complete on $M_H$ with respect to any $\Pi_{t_0}$, $t_0\not\in E\cap\mathbb{R}^2$;
\item[$4)$] moreover, $p^*\big(Z^{\{\Pi'_t\}}(p(U))\big)=\operatorname{Span}\big(\bigcup_{t\not=0}\mu_t^*\big(Z^{\Pi_{\mathfrak{g}^*}}(\mu_t(U))\big)\big)$.
 \end{enumerate}
 Here $\operatorname{ind} \mathfrak{g}$, the index of the Lie algebra $\mathfrak{g}$, is the codimension of a coadjoint orbit of maximal dimension, i.e., $\operatorname{ind} \mathfrak{g}=\dim\mathfrak{g}-\operatorname{rank} \Pi_{\mathfrak{g}^*}$.
\end{Theorem}

\begin{proof} The $G$-invariance of the set $U$ follows from the well-known fact that the Poisson property of the moment map $\mu_t$ is equivalent to its $G$-equivariance (with respect to the coadjoint action of $G$ on $\mathfrak{g}^*$), which implies the $G$-invariance of $\mu_t^*(\operatorname{Sing}\mathfrak{g}^*)$.

The so-called ``bifurcation lemma'' says that for any $x\in M_H$ the image \mbox{$(\mu_t)_*(T_xM_H)\subset\mathfrak{g}^*$} coincides with the annihilator in $\mathfrak{g}^*$ of the Lie algebra of the isotropy group $G^x$ of $x$~\cite[Proposition~4.5.12]{ortegar}. Since this algebra vanishes by Assumption~(b), $\operatorname{rank} \mu_t(x)=\dim\mathfrak{g}^*$ and the image~$\mu_t(M_H)$ contains an open subset of $\mathfrak{g}^*$. The set $\operatorname{Sing}\mathfrak{g}^*$ is algebraic and its complement in~$\mathfrak{g}^*$ is open and dense, hence Assumption (c) guarantees that the set~$U$ is also open and dense.

To prove item 2 observe that for any $t\not=t_i$, $i=1,\dots,s$ and any $x\in M_H$, by the holomorphic version of the bifurcation lemma and by a simple algebraic fact (Lemma \ref{lemM} below) $\operatorname{corank}(({\Pi}^c_t)')_{x}=\operatorname{corank} (\Pi_{(\mathfrak{g}^\mathbb{C})^*})_{\mu^c_t(x)}$. Here $(({\Pi}^c_t)')_{x}$ is the restriction of the bivector $({\Pi}^c_t)_{x}$ treated as a bilinear skew-symmetric form on $T^*_x {M}^c$ to the annihilator $(T_x\mathcal{O})^\circ \subset T^*_x {M}^c$ of the tangent space $T_x\mathcal{O}$ to the $\mathfrak{g}^\mathbb{C}$-orbit $\mathcal{O}$ passing through $x$ and it is known that the space~$T_x\mathcal{O}$ is the skew-orthogonal complement to the tangent space through~$x$ of the fiber of the moment map~$\mu^c_t$.

Hence, if moreover $x\in U$, then $\operatorname{corank}_\mathbb{R}(\Pi_t')_{p(x)}=\operatorname{corank}_\mathbb{C}(({\Pi}^c_t)')_{p(x)}=\operatorname{ind} \mathfrak{g}^\mathbb{C}=\operatorname{ind} \mathfrak{g}$.
Therefore the reduced Poisson pencil $\{\Pi_t\}$ is Kronecker at $p(x)$ if and only if the corank at $p(x)$ of the reductions $(\Pi_{t_i})'$ of the exceptional Poisson structures $\Pi_{t_i}$, $i=1,\dots,s$, is equal to $\operatorname{ind} \mathfrak{g}$.

Item 3 follows from the well known fact that once we have a pair of Poisson submersions $p_1\colon (M,\Pi)\to(M_1,\Pi_1)$ and $p_2\colon (M,\Pi)\to(M_2,\Pi_2)$ with skew-orthogonal fibers with respect to $\Pi$ and complete families of functions $\mathcal{F}_1$, $\mathcal{F}_2$ on $(M_1,\Pi_1)$, $(M_2,\Pi_2)$ respectively, the family $p_1^*(\mathcal{F}_1)\cup p_2^*(\mathcal{F}_2)$ is complete on $(M,\Pi)$ \cite[Proposition~2.22]{p3}.

The last item is a consequence of another well known fact that $p_1^*\big(Z^{\Pi_1}\big)=p_2^*\big(Z^{\Pi_2}\big)$ \cite[Corollary~2.19]{p3}.
\end{proof}

\begin{Lemma}\label{lemM}Let $V$ be a vector space over $\mathbb{K}$ and $\omega\colon V\times V\to \mathbb{K}$ a nondegenerate skew-symmetric bilinear form. Denote by $\Pi\colon V^*\times V^*\to\mathbb{K}$ its inverse bivector. Let $V_1,V_2 \subset V$ be two vector subspaces being orthogonal complements of each other with respect to~$\omega$. Then the restrictions of $\Pi$ to the subspaces $W_1:=V_1^\circ \subset V^*$ and $W_2:=V_2^\circ \subset V^*$ have the same coranks.
\end{Lemma}

\begin{proof} Indeed, since $W_1$ and $W_2$ are mutual orthogonal complements with respect to $\Pi$, we have $\ker(\Pi|_{W_1\times W_1})=W_1\cap W_2=\ker(\Pi|_{W_2\times W_2})$.
\end{proof}

\section[Bi-Poisson structures on cotangent bundles related to Nijenhuis $(1,1)$-tensors]{Bi-Poisson structures on cotangent bundles\\ related to Nijenhuis $\boldsymbol{(1,1)}$-tensors}\label{sec:Bi-Poisson}

\begin{Definition}\label{defi2}Let $Q$ be a manifold and $X\in \Gamma(TQ)$ be a vector field on $Q$. Then the formula $\widetilde X:=\Pi\big(\overline X\big)$, where $\Pi=\omega^{-1}=\partial_q\wedge\partial_p$ is the canonical nondegenerate Poisson structure on $T^*Q$ inverse to the canonical symplectic form $\omega={\rm d}p\wedge {\rm d}q$ and $\overline X$ stands for the linear function on $T^*Q$ corresponding to $X$, gives a vector field $\widetilde X\in\Gamma(TT^*Q)$ which will be called the {\em cotangent lift} of~$X$. The local characterization in the canonical $(q,p)$-coordinates is as follows: if $X=X^i(q)\frac{\partial }{\partial q^i}$, then $\overline{X}=X^i(q)p_i$ and $\widetilde X=X^i(q)\frac{\partial }{\partial q^i}-p_j\frac{\partial X^j(q)}{\partial q^i}\frac{\partial }{\partial p_i}$.
\end{Definition}

\begin{Remark}\label{reem} One can also describe the Hamiltonian function $\overline{X}$ as the evaluation $\theta\big(\widetilde{X}\big)$ of the canonical Liouville 1-form $\theta=p_i\,{\rm d}q^i$ on $\widetilde{X}$.
\end{Remark}

\begin{Remark}\label{refund}In particular, if a Lie group $G$ with $\operatorname{Lie}(G)=\mathfrak{g}$ acts on a manifold $Q$ and $X_\xi$ is the fundamental vector field of this action corresponding to an element $\xi\in\mathfrak{g}$, then $\widetilde X_\xi$ is the corresponding fundamental vector field of the extended cotangent action of~$G$ on~$T^*Q$.
\end{Remark}

Let $(M,\omega)$ be a symplectic manifold, $\omega_1$ another symplectic form on $M$. Then $\omega$, $\omega_1$ are Poisson compatible (i.e., the Poisson structures $\Pi:=\omega^{-1}$, $\Pi_1:=\omega_1^{-1}$ are compatible) if and only if the $(1,1)$-tensor $\widetilde{N}:=\Pi_1\circ\omega\colon TM\to TM$ is Nijenhuis (cf.\ \cite[Proposition~7.1]{Magri97}). Assume this is the case. Let $\Pi^\lambda= \Pi_1- \lambda \Pi $. We have $\Pi^\lambda =\big(\widetilde{N} - \lambda I \big) \Pi $,
 therefore $ \operatorname{im} \Pi^\lambda = \operatorname{im} \big(\widetilde{N} - \lambda I \big) \Pi = \operatorname{im} \big(\widetilde{N} - \lambda I \big) $ since $\Pi$ is nondegenerate. This gives us relation between characteristic distributions of Poisson structures and eigendistributions of a~Nijenhuis tensor. In particular, we have proved the following lemma.

\begin{Lemma}\label{symplf}Retaining the assumptions above assume additionally that $\widetilde{N}$ is semisimple and has constant eigenvalues $\lambda_1,\dots,\lambda_s$, $\lambda_i\not=\lambda_j$, $i\not=j$ $($assumed to be real in the real category$)$. Let $\widetilde D_i$, $i=1,\dots,s$, be the eigendistribution corresponding to the eigenvalue~$\lambda_i$. Then the foliation~$\mathcal{F}_i$ tangent to the distribution $\sum\limits_{j\not=i}\widetilde D_j$ $($which is integrable by Lemma~{\rm \ref{L4})} coincides with the symplectic foliation of the degenerate Poisson bivector~$\Pi^{\lambda_i}$.
\end{Lemma}

The following definition is due to F.-J.~Turiel~\cite{t12}.

\begin{Definition}\label{ctgl}Let $Q$ be a manifold and $K\colon TQ\to TQ$ be a $(1,1)$-tensor. Define its {\em cotangent lift} $\widetilde{K}\colon TM\to TM$, $M:=T^*Q$, as follows. Let $K^t\colon T^*Q\to T^*Q$ be the map transposed to~$K$ understood as a smooth map $M\to M$, let $\omega$ be the canonical symplectic form on~$T^*Q$, and let $\omega_1:=\big(K^t\big)^*\omega$. Put $\widetilde{K}:=\omega^{-1}\circ\omega_1$.
\end{Definition}

If $\{q^i\}$ is a system of local coordinates on $Q$ and $K=K^i_j (q) \frac{\partial}{\partial q^i } \otimes {\rm d}q^j $, then in the corresponding coordinates $\big(q^i, p_i\big)$ on $T^*Q$ we have
\begin{gather*}
\widetilde{K}= K^i_j (q) \left( \frac{\partial}{\partial q^i } \otimes {\rm d}q^j + \frac{\partial}{\partial p_j } \otimes {\rm d}p_i \right)
+ p_k \left( \frac{\partial K^k_j }{\partial q^i} - \frac{\partial K^k_i }{\partial q^j} \right) \frac{\partial}{\partial q^j}\otimes {\rm d}q^i
.\end{gather*}
Obviously, if $K$ is a fiberwise invertible $(1,1)$-tensor, then $\widetilde{K^{-1}}=\widetilde{K}^{-1}$.

\begin{Lemma}[\cite{t12}]\label{tors} $T_K=0\Longleftrightarrow T_{\widetilde{K}}=0$.
\end{Lemma}

In particular, we have the following statement.

\begin{Lemma}\label{constt}Let $N\colon TQ\to TQ$ be a fiberwise invertible Nijenhuis $(1,1)$-tensor. Then the pair of bivectors $(\Pi,\Pi_1)$, where $\Pi:=\omega^{-1}$ is the canonical Poisson bivector on $T^*Q$, $\Pi_1:=\widetilde{N}\circ\Pi$, and~$\omega $ is the canonical symplectic form on~$T^*Q$, is a pair of compatible Poisson bivectors on~$T^*Q$.
\end{Lemma}

\begin{proof} Obviously, $\widetilde{N}\circ\Pi=\widetilde{N}\circ \omega^{-1}=\big(\omega\circ \widetilde{N}^{-1}\big)^{-1}=\big(\omega\circ \widetilde{N^{-1}}\big)^{-1}=\big(\big((N^{-1})^t\big)^*\omega\big)^{-1}=N^t_*\omega^{-1}=N^t_*\Pi$, hence~$\Pi_1$ is a Poisson bivector. Since $\widetilde{N}$ is a Nijenhuis tensor, $\Pi$ and $\Pi_1$ are compatible.
\end{proof}

From now on we assume that $N$ is an invertible semisimple Nijenhuis $(1,1)$-tensor with constant eigenvalues $\lambda_1,\dots,\lambda_s$, $\lambda_i\not=\lambda_j$, $i\not=j$ (which are real in the real category) of multiplicities $k_1,\dots,k_s$ respectively and let $D_i \subset TQ$ to be the eigendistribution corresponding to the eigenvalue~$\lambda_i$. We also denote by $\widetilde{ D}_i \subset TM$ the eigendistribution of the $(1,1)$-tensor $\widetilde{N}$ corresponding to the eigenvalue $\lambda_i$ (of multiplicity~$2k_i$).

Let $L(N)$ stand for the Lie algebra of vector fields on $Q$ preserving $N$ (i.e., $V\in L(N)$ if and only if $\mathcal{L}_VN=0$) and let $\theta=\theta_1+\cdots+\theta_s$ be the decomposition of the canonical Liouville one-form~$\theta$ on $M:=T^*Q$ related to the decomposition $TM=\widetilde{ D}_1\oplus\cdots\oplus \widetilde{ D}_s$, i.e., $\theta_i|_{\widetilde{ D}_i}=\theta|_{\widetilde{ D}_i}$ and $\theta_i\big(\sum\limits_{j\not=i}\widetilde{ D}_j\big)=0$.

\begin{Lemma}\label{L56} Retain the assumptions above. Then the following statements hold.
\begin{enumerate}\itemsep=0pt
\item[$1.$] For any $i\in\{1,\dots,s\}$ the leaves of the symplectic foliation $\mathcal{F}_i$ of the Poisson bivector $\Pi^{\lambda_i}:=\Pi_1-\lambda_i\Pi$ are all of the same dimension and have codimension $2k_i$. For any such leaf $F$ its image $\pi(F)$ under the canonical projection $\pi\colon T^*Q\to Q$ is a leaf of the foliation tangent to the distribution $\check D_i:=\sum\limits_{j\not=i}D_j$ $($which is integrable by Lemma~{\rm \ref{L4})}. The leaf $\pi(F)$ is of codimension~$k_i$ in~$Q$. For any leaf $F_0$ of the foliation tangent to $\check D_i$ the set $\pi^{-1}(F_0)$ is a Poisson submanifold of the Poisson manifold $\big(T^*Q,\Pi^{\lambda_i}\big)$.
\item[$2.$] The vector fields from $L(N)$ tangent to the distribution $D_i$ form an ideal $L_i$ of the Lie algebra $L(N)$ and there is a direct decomposition $L(N)=L_1\oplus\cdots\oplus L_s$.
\item[$3.$] For any $V\in L(N)$ and $i\in\{1,\dots,s\}$ the function $f^i_V:=\theta_i\big(\widetilde{ V}\big)\in\mathcal{E}(T^*Q)$, where $\widetilde{ V}$ is the cotangent lift of $V$ $($see Definition~{\rm \ref{defi2})}, is a Casimir function of the bivector $\Pi^{\lambda_i}$. In particular, since $f^i_V$ linearly depends on $V\in L(N)$, for any leaf $F$ of the foliation $\mathcal{F}_i$ the formula $V\mapsto\Phi^i_F(V):=f^i_V|_F$ defines a linear functional on~$L(N)$.
 \item[$4.$] If $V\in L(N)$ is tangent to the leaf $\pi(F)$, where $F$ is a symplectic leaf of the Poisson bivector $\Pi^{\lambda_i}$, then $\Phi^i_F(V)=0$.
 \item[$5.$] Given any leaf $F_0 \subset Q$ of the foliation tangent to the distribution~$\check D_i$, a vector field \smash{$V\in L(N)$} is tangent to $F_0$ if and only if $\widetilde{V}$ is tangent to $\pi^{-1}(F_0)$.
\end{enumerate}
 \end{Lemma}

 \begin{proof} Recall that $k_i$ is the multiplicity of the eigenvalue $\lambda_i$, $i=1,\dots,s$. By Lemma \ref{L4} in a~vicinity of every point on $Q$ there exist a system of local coordinates $\big(q^{j^1_1},\dots,q^{j^1_{k_1}},\dots,q^{j^s_1},\dots,\allowbreak q^{j^s_{k_s}}\big)$, where $\big(j^1_1,\dots,j^1_{k_1},\dots,j^s_1,\dots,j^s_{k_s}\big)=(1,\dots,\dim Q)$, such that the eigendistribution~$D_i$ is spanned by the vector fields $\frac{\partial }{\partial q^{j^i_1}},\dots,\frac{\partial }{\partial q^{j^i_{k_i}}}$. Then $\widetilde{N}=\sum_i\lambda_i\big( \sum\limits_{n=1}^{k_i}\big(\frac{\partial}{\partial q^{j^i_n } }\otimes {\rm d}q^{j^i_n }+ \frac{\partial}{\partial p_{j^i_n } }\otimes {\rm d}p_{j^i_n}\big)\big)$ and by Lemma~\ref{symplf} the tangent space to the symplectic foliation of $\Pi^{\lambda_i}$ is generated by the vector fields $\frac{\partial }{\partial q^l}$, $\frac{\partial }{\partial p_m}$, $l,m\not\in\big\{j^i_1,\dots,j^i_{k_i}\big\}$ (the corresponding Casimir functions are $q^{j^i_n }$, $p_{j^i_n }$, $n=1,\dots,k_i$). On the other hand, the tangent distribution to the leaves of $\sum\limits_{j\not=i}D_j$ is spanned by the vector fields $\frac{\partial }{\partial q^l}$, $l\not\in\big\{j^i_1,\dots,j^i_{k_i}\big\}$. This proves the first assertion of the lemma.

To prove item~2 notice that, given a vector field $V=V^l(q)\frac{\partial }{\partial q^l}$, the equality $\mathcal{L}_VN=0$ holds if and only if $[V,NX]=N[V,X]$ for any vector field $X$. Substituting $X=\frac{\partial }{\partial q^{j^i_n}}$ to the last equality we get $\lambda_i\big[V,\frac{\partial }{\partial q^{j^i_n}}\big]=N\big[V,\frac{\partial }{\partial q^{j^i_n}}\big]$, which means that $\big[V,\frac{\partial }{\partial q^{j^i_n}}\big]$ is an eigenvector of $N$ corresponding to $\lambda_i$. Hence $\big[V,\frac{\partial }{\partial q^{j^i_n}}\big]$ is expressed as a linear combination of $\frac{\partial }{\partial q^{j^i_1}},\dots,\frac{\partial }{\partial q^{j^i_{k_i}}}$. In other words, the coefficients $V^{j^i_1},\dots,V^{j^i_{k_i}}$ depend only on the coordinates $q^i:=\big(q^{j^i_1},\dots,q^{j^i_{k_i}}\big)$ for any $i$.

For the proof of item~3 observe that $\theta_i=\sum\limits_{n=1}^{k_i}p_{j^i_n}\,{\rm d} q^{j^i_n}$ in the coordinates mentioned and the evaluation $f^i_V$ of this form on the cotangent lift
\begin{gather} \label{o}
 \widetilde{V}=V^l(q)\frac{\partial }{\partial q^l}-p_l\frac{\partial V^l(q)}{\partial q^n}\frac{\partial }{\partial p_n}
\end{gather}
of a vector field $V\in L(N)$ is equal to
\begin{gather*}
f^i_V=\sum_{n=1}^{k_i}V^{j^i_n}\big(q^i\big)p_{j^i_n}.
\end{gather*}
Any leaf $F_0$ of the foliation tangent to $\check D_i$ is given in these coordinates by the equations $q^{j^i_n}=c^{j^i_n}$, $n=1,\dots,k_i$, and any symplectic leaf $F \subset\pi^{-1}(F_0)$ of the foliation $\mathcal{F}_i$ by the equations $q^{j^i_n}=c^{j^i_n}$, $p_{j^i_n}=C_{j^i_n}$, $n=1,\dots,k_i$, whose right hand sides are some constants. This proves item~3.

If a vector field $V\in L(N)$ is tangent to $\pi(F)$, then
\begin{gather*}
 V^{j^i_n}\big(c^i\big)=0,\qquad n=1,\dots,k_i,
\end{gather*}
where we put $c^i:=\big(c^{j^i_1},\dots,c^{j^i_{k_i}}\big)$, in particular $\Phi^i_F=f^i_V|_F=0$.

The last item follows easily from formula (\ref{o}).
\end{proof}

\begin{Lemma}\label{lemk}Retaining the assumptions of the preceding lemma assume that a transitive left action $\rho\colon \mathfrak{g}\to\Gamma(TQ)$, $\xi\mapsto V_\xi$, of a Lie algebra $\mathfrak{g}$ on $Q$ is given such that $\rho$ preserves $N$, i.e., $\rho(\mathfrak{g}) \subset L(N)$.
 Denote by $\tilde\rho$ the extended cotangent action, $\tilde\rho(\xi)=\widetilde{\rho(\xi)}$, $\xi\in\mathfrak{g}$. $($Note that $\rho$ is an antihomomorphism, the map $V\mapsto \widetilde{V}$ is a homomorphism, hence $\tilde\rho$ is an antihomomorphism, i.e., a~left action.$)$ Given a leaf $F$ of the symplectic foliation $\mathcal{F}_i$, $i\in\{1,\dots,s\}$, let $\varphi^i_F:=\Phi^i_F\circ\rho\in\mathfrak{g}^*$ be the linear functional induced on~$\mathfrak{g}$ by the functional $\Phi_F^i\in(L(N))^*$ from Lemma~{\rm \ref{L56}(3)} and let~$\mathfrak{g}_F$ stand for the stabilizer algebra of~$F$, i.e., the set of elements $\xi\in\mathfrak{g}$ such that $\tilde{\rho}(\xi)$ is tangent to~$F$. Let $p_i\colon \Gamma(TQ)\to \Gamma(D_i)$ be the projection related to the decomposition $TQ=D_1\oplus\cdots\oplus D_s$. Then
 \begin{enumerate}\itemsep=0pt
 \item[$1)$] for any $i\in\{1,\dots,s\}$ the map $L(N)\to L_r$ induced by the projection $p_i$ is a homomorphism of Lie algebras; in particular $\rho_i\colon \mathfrak{g}\to\Gamma(TQ)$, where $\rho_i:=p_i\circ\rho$, is a left action of the Lie algebra~$\mathfrak{g}$;
 \item[$2)$] the action $\tilde\rho$ is Hamiltonian with respect to the Poisson structure $\Pi^\lambda$ for any $\lambda\not=\lambda_i$, $i=1,\dots,s$ with the momentum map $\mu_\lambda\colon T^*Q\to \mathfrak{g}^*$ given by
 \begin{gather}\label{mom0}
 \langle\mu_\lambda(x),\xi\rangle=(\psi(\lambda)^*\theta)(\tilde\rho(\xi))(x),\qquad \xi\in \mathfrak{g},
 \end{gather}
 where $\theta$ is the canonical Liouville $1$-form on $T^*Q$ and $\psi(\lambda)$ is the diffeomorphism of $T^*Q$ given by $\big((N-\lambda I)^t\big)^{-1}$ $($we used the notation $(\cdot)^t$ for the transposed map$)$; equivalently, $\langle\mu_\lambda(x),\xi\rangle=\sum\limits_{i=1}^s\frac{f^i_{\rho(\xi)}(x)}{\lambda_i-\lambda}$, see Lemma~{\rm \ref{L56}(3)} for the definition of~$f^i_{V}$; moreover,
 \begin{gather}\label{mom}
 \mu_\lambda=\mu_{\rm can}\circ\psi(\lambda),
 \end{gather}
 where $\mu_{\rm can}$ is the moment map corresponding to the canonical Poisson bivector~$\Pi$;
 \item[$3)$] given a leaf $F$ of the symplectic foliation $\mathcal{F}_i$, the restricted action
 of $\mathfrak{g}_F$ on $F$ is Hamiltonian with respect to the restriction of the Poisson structure $\Pi^{\lambda_i}$ to $F$ with the momentum map $\mu_{\lambda_i}^F\colon F\to \mathfrak{g}^*$ given by $\big\langle\mu_{\lambda_i}^F(x),\xi\big\rangle=(\psi_{\lambda_i}^*\theta)(\tilde\rho(\xi))(x)$, $\xi\in \mathfrak{g}_F$, where~$\theta$ is the canonical Liouville $1$-form on $T^*Q$ and $\psi_{\lambda_i}$ is the smooth map of $T^*Q$ given by $\psi_{\lambda_j}|_{D^*_i}=\big((N-\lambda_i I)^t\big)^{-1}|_{D^*_i}$, $j\not=i$, $\psi_{\lambda_i}|_{D^*_i}=0$; here $T^*Q=D^*_1\oplus\cdots\oplus D^*_s$ is the decomposition corresponding to $TQ=D_1\oplus\cdots\oplus D_s$;\footnote{Here an equivalent description of the momentum map similar to that from item~2 is also possible: $\langle\mu_{\lambda_i}^F(x),\xi\rangle=\sum\limits_{j=1}^s\frac{f^j_{\rho(\xi)}(x)}{\lambda_j-\lambda_i}$ (note that the $i$-th term in the sum is correctly defined since $f^i_{\tilde\rho(\xi)}$ vanishes for $\xi\in \mathfrak{g}_F$, cf.\ Lemma~\ref{L56}(4)).}
\item[$4)$] the cotangent extension $\tilde\rho_i$, $\tilde\rho_i(\xi):=\widetilde{\rho_i(\xi)}$, of the action $\rho_i$ defined in item~$1$ is Hamiltonian with respect to the canonical Poisson bivector $\Pi$ with the momentum map $\nu_i\colon T^*Q\to\mathfrak{g}^*$, $\langle\nu_i(x),\xi\rangle={f^i_{\rho(\xi)}(x)}$;
 \item[$5)$] for any leaf $F_0 \subset Q$ of the foliation tangent to the distribution $\check D_i$ its stabilizer algebra $\mathfrak{g}_{F_0}$ with respect to the action $\rho$, i.e., the set of $\xi\in \mathfrak{g}$ such that $\rho(\xi)$ is tangent to $F_0$, coincides with the stabilizer algebra of the submanifold $\pi^{-1}(F_0)$ with respect to the action $\tilde\rho$, i.e., the set of $\xi\in \mathfrak{g}$ such that $\tilde\rho(\xi)$ is tangent to $\pi^{-1}(F_0)$;
 \item[$6)$] the following inclusion holds: $\nu_i\big(\pi^{-1}(F_0)\big)\subset \mathfrak{g}_{F_0}^\bot\cong (\mathfrak{g}/\mathfrak{g}_{F_0})^*$;
\item[$7)$] moreover, the relation $F\mapsto\varphi_F^i=\nu_i|_F$ is an $\mathfrak{g}_{F_0}$-equivariant one-to-one correspondence between the symplectic leaves $F$ of $\Pi^{\lambda_i}$ such that $\pi(F)=F_0$ and linear functionals from a $k_i$-dimensional linear subspace in $(\mathfrak{g}/\mathfrak{g}_{F_0})^*$ $($which in fact coincides with $(\mathfrak{g}/\mathfrak{g}_{F_0})^*$, see Lemma~{\rm \ref{stab}(4))}.
\item[$8)$] the stabilizer algebra $\mathfrak{g}_F \subset \mathfrak{g}$ of a leaf $F \subset \pi^{-1}(F_0)$ with respect to~$\tilde\rho$ is equal to the stabilizer algebra of the functional $\varphi_F^i\in(\mathfrak{g}/\mathfrak{g}_{F_0})^*$ with respect to the action of~$\mathfrak{g}_{F_0}$.
 \end{enumerate}
\end{Lemma}

\begin{proof} The claim of item~1 follows from the fact that each $L_i$ is an ideal in $L(N)$ (see Lemma~\ref{L56}(2)).

To prove items~2 and~3 use coordinates from the proof of the previous lemma. We have the following formulas: $\Pi=\sum\limits_i \sum\limits_{n=1}^{k_i}\frac{\partial }{\partial p_{j^i_n}}\wedge \frac{\partial}{\partial q^{j^i_n } }$, $\Pi_1=\sum\limits_i\lambda_i\big( \sum\limits_{n=1}^{k_i}\frac{\partial }{\partial p_{j^i_n }}\wedge \frac{\partial}{\partial q^{j^i_n } }\big)$,
$\Pi^{\lambda}=\sum\limits_i(\lambda_i-\lambda)\big( \sum\limits_{n=1}^{k_i}\frac{\partial }{\partial p_{j^i_n }}\wedge \frac{\partial}{\partial q^{j^i_n } }\big)$,
and $\Pi^{\lambda_j}=\sum\limits_{i\not=j}(\lambda_i-\lambda_j)\big( \sum\limits_{n=1}^{k_i}\frac{\partial }{\partial p_{j^i_n }}\wedge \frac{\partial}{\partial q^{j^i_n } }\big)$.
For the vector field $V=V_\xi=\sum\limits_i \sum\limits_{n=1}^{k_i}V_\xi^{j^i_n }\big(q^i\big)\frac{\partial}{\partial q^{j^i_n } }$ its cotangent lift $\widetilde{V}=\widetilde{V}_\xi$ takes the form
\begin{gather*}
\widetilde{V}=\sum_i \sum_{n=1}^{k_i}V_\xi^{j^i_n }\big(q^i\big)\frac{\partial}{\partial q^{j^i_n } }-\sum_{i} \sum_{n,m=1}^{k_i}p_{j^i_n }\frac{\partial V_\xi^{j^i_n }\big(q^i\big)}{\partial q^{j^i_m }}\frac{\partial }{\partial p_{j^i_m }}
\end{gather*}
and is a Hamiltonian vector field with respect to~$\Pi$: $\widetilde{V}=\Pi(H_\xi)$, where $H_\xi=\theta\big(\widetilde{V}\big)=\sum\limits_i \sum\limits_{n=1}^{k_i}V_\xi^{j^i_n }\big(q^i\big)p_{j^i_n }$ (cf.\ Remark~\ref{reem}). On the other hand, obviously,
$\widetilde{V}=\Pi^\lambda(H^\lambda_\xi)$, where we put
\begin{gather*}
H^\lambda_\xi=\sum_i \sum_{n=1}^{k_i}V_\xi^{j^i_n }\big(q^i\big)p_{j^i_n }/(\lambda_i-\lambda).
\end{gather*}
In fact, the functions $H^\lambda_\xi$ are global and correctly defined (i.e., they do not depend on the choices of local coordinates), which can be seen from the equality $H^\lambda_\xi=(\psi(\lambda)^*\theta)\big(\widetilde{V}_\xi\big)$. Yet another description of the function $H^\lambda_\xi$ is as follows: $H^\lambda_\xi=\sum_i f^i_{V_\xi}/(\lambda_i-\lambda)$ (see Lemma~\ref{L56}(3) for the definition of~$f^i_{V_\xi}$).

Using the equality $V_{[\xi,\zeta]}=[V_\xi,V_\zeta]$ we get
\begin{gather*}
 \Pi^{\lambda} \big(H^\lambda_\xi\big)H^\lambda_\zeta - H^\lambda_{[\xi,\zeta] } =
 \sum_i \sum_{n,m=1}^{k_i}V_\xi^{j^i_n }\big(q^i\big)\frac{\partial V_\zeta^{j^i_m }\big(q^i\big)}{\partial q^{j^i_n } }\frac{p_{j^i_m }}{\lambda_i-\lambda}\\
\qquad {} -\sum_{i} \sum_{n,m=1}^{k_i}p_{j^i_n }\frac{\partial V_\xi^{j^i_n }\big(q^i\big)}{\partial q^{j^i_m }}V_\zeta^{j^i_m }\big(q^i\big)\frac{1}{\lambda_i-\lambda}
 -\sum_i \sum_{n=1}^{k_i}V_{[\xi,\zeta]}^{j^i_n }\big(q^i\big)\frac{p_{j^i_n }}{\lambda_i-\lambda} = 0,
\end{gather*}
which proves the hamiltonicity of $\tilde\rho$ with respect to $\Pi^\lambda$, $\lambda\not=\lambda_i$.

Formula (\ref{mom}) is a consequence of (\ref{mom0}) as $\langle\mu_{\rm can}(x),\xi\rangle=\theta(\tilde\rho(\xi))(x)$.

Now assume that $V=V_\xi$ is tangent to the symplectic leaf $F$ given in the local coordinates by the equations $q^{j^i_n}=\mathrm{const}$, $p_{j^i_n}=\mathrm{const}$, $n=1,\dots,k_i$. Then by~(\ref{o}) we get
\begin{gather*}
\widetilde{V}|_F=\sum_{l\not=i} \sum_{n=1}^{k_l}V_\xi^{j^l_n }\big(q^l\big)\frac{\partial}{\partial q^{j^l_n } }-\sum_{l\not=i} \sum_{n,m=1}^{k_l}p_{j^l_n }\frac{\partial V_\xi^{j^l_n }\big(q^l\big)}{\partial q^{j^l_m }}\frac{\partial }{\partial p_{j^l_m }}=
\Pi^{\lambda_i}(H_\xi)
 \big|_F ,
\end{gather*}
where $H_\xi=\sum\limits_{l\not=i} \sum\limits_{n=1}^{k_l}V_\xi^{j^l_n }\big(q^l\big)p_{j^l_n }/(\lambda_l-\lambda_i)$. The function $H_\xi$ is global and correctly defined for any~$\xi$ as $H_\xi=\sum\limits_{l\not=i} f^l_{V_\xi}/(\lambda_l-\lambda_i)$ and
\begin{gather*}
 \Pi^{\lambda_i} (H_\xi)H_\zeta - H_{[\xi,\zeta] } =
 \sum_{l\not=i} \sum_{n,m=1}^{k_l}V_\xi^{j^l_n }\big(q^l\big)\frac{\partial V_\zeta^{j^l_m }\big(q^l\big)}{\partial q^{j^l_n } }\frac{p_{j^l_m }}{\lambda_l-\lambda_i}\\
 \qquad{} -\sum_{l\not=i} \sum_{n,m=1}^{k_l}p_{j^l_n }\frac{\partial V_\xi^{j^l_n }\big(q^l\big)}{\partial q^{j^l_m }}V_\zeta^{j^l_m }\big(q^l\big)\frac{1}{\lambda_l-\lambda_i}
 -\sum_{l\not=i} \sum_{n=1}^{k_l}V_{[\xi,\zeta]}^{j^l_n }\big(q^l\big)\frac{p_{j^l_n }}{\lambda_l-\lambda_i} = 0.
\end{gather*}
Since $F$ is a Poisson submanifold with respect to $\Pi^{\lambda_i}$, we have
\begin{gather*} \{H_\xi|_F,H_\zeta|_F\}_{\Pi^{\lambda_j}|_F}= \{H_\xi,H_\zeta\}_{\Pi^{\lambda_j}}|_F= H_{[\xi,\zeta] }|_F.\end{gather*}

To prove item~4 notice that, if $V_\xi=\sum\limits_i \sum\limits_{n=1}^{k_i}V_\xi^{j^i_n }\big(q^i\big)\frac{\partial}{\partial q^{j^i_n } }$, $\xi\in\mathfrak{g}$, is the fundamental vector field of the action $\rho$, then $\rho_i(\xi)=\sum\limits_{n=1}^{k_i}V_\xi^{j^i_n }\big(q^i\big)\frac{\partial}{\partial q^{j^i_n } }$ is the fundamental vector field of the action $\rho_i$. Its cotangent lift $\tilde{\rho}_i(\xi)$ is a Hamiltonian vector field with respect to $\Pi$ with the Hamiltonian function $H_\xi^i:=f^i_{\rho_i(\xi)}=f^i_{V_\xi}=\sum\limits_{n=1}^{k_i}V_\xi^{j^i_n }\big(q^i\big)p_{j^i_n}$. Now it remains to use the equality $\rho_i([\xi,\zeta])=[\rho_i(\xi),\rho_i(\zeta)]$, which implies
\begin{gather*}
 \Pi \big(H^i_\xi\big)H^i_\zeta - H^i_{[\xi,\zeta] } =
 \sum_{n,m=1}^{k_i}V_\xi^{j^i_n }\big(q^i\big)\frac{\partial V_\zeta^{j^i_m }\big(q^i\big)}{\partial q^{j^i_n } }{p_{j^i_m }}- \sum_{n,m=1}^{k_i}p_{j^i_n }\frac{\partial V_\xi^{j^i_n }\big(q^i\big)}{\partial q^{j^i_m }}V_\zeta^{j^i_m }\big(q^i\big)\\
 \hphantom{\Pi \big(H^i_\xi\big)H^i_\zeta - H^i_{[\xi,\zeta] } =}{} - \sum_{n=1}^{k_i}V_{[\xi,\zeta]}^{j^i_n }\big(q^i\big){p_{j^i_n }} = 0.
\end{gather*}

Item~5 follows from Lemma \ref{L56}(5) and item~6 follows from Lemma~\ref{L56}(4) in view of the fact that $\pi^{-1}(F_0)$ is foliated by the symplectic leaves of the Poisson bivector $\Pi^{\lambda_i}$ (see Lemma~\ref{L56}(1)) and from the equality $\Phi^i_F(\rho_i(\xi)):=f^i_{\rho_i(\xi)}|_F$, where $F$ is any such leaf.

To prove item~7 first notice that the $\mathfrak{g}_{F_0}$-equivariance follows from $\mathfrak{g}$-equivariance of the moment map $\nu_i$. Now recall (see the proof of Lemma \ref{L56}) that
\begin{gather*}
\varphi_F^i(\xi)=\sum_{n=1}^{k_i}V_\xi^{j^i_n}\big(c^i\big)C_{j^i_n}, \qquad \xi\in\mathfrak{g},
\end{gather*}
where the constants $c^{i}$, $C_{j^i_n}$ specify the particular leaf~$F$ and $V_\xi^{j^i_n}$ are the coefficients of the fundamental vector field $\rho(\xi)=V_\xi^l(q)\frac{\partial }{\partial q^l}$.

Now fix a leaf $F_0 $ of the foliation tangent to a distribution $\check D_i$, i.e., fix constants $\big(c^{i}\big)$. For any $\xi\in\mathfrak{g}$ we have a linear map\footnote{Note that although the range of constants~$c^i$ is bounded by that of the local coordinates~$q^i$, the constants~$C_{j^i_n}$ can take any value.}
\begin{gather*}
\big(C_{j^i_1},\dots,C_{j^i_{k_i}}\big)\mapsto \sum_{n=1}^{k_i}V_\xi^{j^i_n}\big(c^i\big)C_{j^i_n},
\end{gather*}
expressing the correspondence $F\mapsto \varphi_F^i(\xi)$, where $k_i=\operatorname{corank} \check D_i$. Thus the claim of item~7 is equivalent to the nondegeneracy of the following matrix
\begin{gather*}
\left[
 \begin{matrix}
 V_{\xi_1}^{j^i_1}\big(c^i\big) & \cdots & V_{\xi_1}^{j^i_{k_i}}\big(c^i\big) \\
 \vdots & & \vdots \\
 V_{\xi_{k_i}}^{j^i_1}\big(c^i\big) & \cdots & V_{\xi_{k_i}}^{j^i_{k_i}}\big(c^i\big)
 \end{matrix}
\right],
\end{gather*}
where $\xi_1,\dots,\xi_{k_i}\in\mathfrak{g}$ are linearly independent elements not belonging to $\mathfrak{g}_{F_0}$. In turn, the nondegeneracy of this matrix follows from the fact that $\mathfrak{g}$ acts transitively on~$G/K$ and, as a~consequence, on the space of leaves of the foliation tangent to the distribution~$\check D_i$.

Finally the last item follows from item~7.
\end{proof}

Now we apply the preceding results to homogeneous spaces. Let $G/K$ be a homogeneous space and let $N$ be an $G$-invariant semisimple Nijenhuis $(1,1)$-tensor on $G/K$ with the real spectrum $\{\lambda_1,\dots,\lambda_s\}$. Then by Theorem \ref{T1} there exists a decomposition $\mathfrak{g} = \mathfrak{g}_1 + \cdots+ \mathfrak{g}_s $ to the sum of subspaces such that
\begin{enumerate}\itemsep=0pt
\item[1)] $ \forall_{i,j\in \{1,\dots,s\}, i\not=j }$ $\mathfrak{g}_i\cap\mathfrak{g}_j = \mathfrak{k}$;
\item[2)] $ \forall_{i,j\in \{1,\dots,s\} }$ $\mathfrak{g}_i+\mathfrak{g}_j $ are Lie subalgebras in $\mathfrak{g}$;
\item[3)] the decomposition above induces the decomposition $T(G/K)=D_1\oplus\cdots\oplus D_s$ to integrable subbundles and $N|_{D_i}=\lambda_i\operatorname{Id}_{D_i}$.
\end{enumerate}
 Write $P\colon G\to G/K$ and $\pi\colon T^*(G/K)\to G/K$ for the canonical projections.

By the construction from the proof of Lemma \ref{L5} the eigendistribution $D_i$ of $N$ corresponding to the eigenvalue $\lambda_i$ is equal $P_*\widehat D_i$, where $\widehat D_i$ is the left invariant distribution on $G$ obtained from the subspace $\mathfrak{g}_i \subset\mathfrak{g}\cong T_eG$. In particular, since $\ker P_*$ is the left invariant distribution obtained from the subspace $\mathfrak{k} \subset\mathfrak{g}_i \subset\mathfrak{g}\cong T_eG$, the rank of~$D_i$, i.e., the multiplicity $k_i$ of the eigenvalue~$\lambda_i$, is equal to $\dim (\mathfrak{g}_i/\mathfrak{k})$.

Denote $\check{\mathfrak{g}}_i:=\sum\limits_{j\not=i}\mathfrak{g}_j$ (this is a Lie subalgebra in~$\mathfrak{g}$ by condition~2) and let~$\check G_i$ be the corresponding subgroup in $G$. By Lemma~\ref{L5} the leaves of the foliation integrating the distribution
$\check D_i:=\sum\limits_{j\not=i}D_j$ are the projections with respect to $P$ of the left cosets $g \check G_i$, $g\in G$. Let $p_i\colon TQ\to D_i$ be the projection related to the decomposition $TQ=D_1\oplus\cdots\oplus D_s$.

\begin{Lemma}\label{stab}Let $N$ be an invertible Nijenhuis $(1,1)$-tensor on a homogeneous space $Q=G/K$ satisfying the assumptions above. Let $\Pi$ be the canonical poisson bivector on $T^*(G/K)$ and $\Pi_1=\widetilde{N}\circ\Pi$ $($see Lemma~{\rm \ref{constt})}. Then
 \begin{enumerate}\itemsep=0pt
 \item[$1)$] for any symplectic leaf $F$ of the Poisson bivector $\Pi^{\lambda_i}:=\Pi_1-\lambda_i\Pi$ there exists an element $g\in G$ such that $\pi(F)=P\big(g\check G_i\big)$; such element $g$ is unique modulo right multiplication by $h\in\check G_i$;
 \item[$2)$] the stabilizer algebra $\mathfrak{g}_{\pi(F)} \subset\mathfrak{g}$ of the leaf $\pi(F)=P\big(g\check G_i\big)$ of the foliation tangent to the distribution $\check D_i$ with respect to the $G$-action on $G/K$ is equal to $\operatorname{Ad}_g\check{\mathfrak{g}}_i$;
 \item[$3)$] the stabilizer algebra $\mathfrak{g}_{F} \subset\mathfrak{g}$ of the leaf $F$ with respect to the extended $G$-action on $T^*(G/K)$ is equal to the stabilizer algebra $\mathfrak{g}^{\varphi_F^i} \subset\operatorname{Ad}_g\check{\mathfrak{g}}_i$ of the functional $\varphi_F^i\in(\mathfrak{g}/\operatorname{Ad}_g\check{\mathfrak{g}}_i)^*$ constructed in Lemma~{\rm \ref{lemk}} by means of an action $\rho$, where we specify $\rho\colon \mathfrak{g}\to\Gamma(T(G/K))$ to be the natural action of the Lie algebra $\mathfrak{g}$ on $G/K$;
 \item[$4)$] if $F_0 \subset G/K$ is a fixed leaf of the foliation tangent to the distribution $\check D_i$, $F_0=P\big(g\check G_i\big)$ $(g$~fixed$)$, the relation $F\mapsto\varphi_F^i$ is an $\operatorname{Ad}_g\check{\mathfrak{g}}_i$-equivariant one-to-one correspondence between the symplectic leaves $F$ of $\Pi^{\lambda_i}$ such that $\pi(F)=F_0$ and linear functionals from $(\mathfrak{g}/\operatorname{Ad}_g\check{\mathfrak{g}}_i)^*$.
 \end{enumerate}
\end{Lemma}

\begin{proof}First and second items are consequences of Lemma \ref{L5} applied to the subalgebra \mbox{$\mathfrak{h}=\check {\mathfrak{g}}_i$}.
Item 3 follows from item 2 and Lemma~\ref{lemk}(8). Item~4 follows from Lemma~\ref{lemk}(7) since $k_i=\dim\mathfrak{g}_i-\dim\mathfrak{k}=\dim (\mathfrak{g}/\operatorname{Ad}_g\check{\mathfrak{g}}_i)^*$.
\end{proof}

\section[Algebraic criterion of kroneckerity in the case of a locally free action]{Algebraic criterion of kroneckerity in the case\\ of a locally free action}\label{s.kro}

The theorem below is the main result of this paper. Let $G$ be a compact Lie group, $K$ its closed subgroup. Assume that the natural action of $G$ on $M=T^*(G/K)$ is generically {\it locally free}, i.e., the stabilizer corresponding to the principal orbit type is finite. Fix such a stabilizer~$H$. In this case
the subset
\begin{gather*}
M_{H}=\big\{x\in M\colon G_x=gHg^{-1}\ \mathrm{for\ some}\ g\in G\big\}
\end{gather*}
of $M$, consisting of all orbits $G\cdot x$ in $M$ isomorphic to $G/H$, is an open and
dense subset of $M$ and the orbit space $M'_H:=M_{H}/G$ is a smooth manifold (cf.~Section~\ref{sec:Bi-Poisson0}). Write $p\colon M_H\to M_H/G$ for the canonical projection.

\begin{Theorem}\label{kro} Let $N$ be an $G$-invariant invertible\footnote{Invertibility can be always achieved by adding the identity operator, which does not change the corresponding pencil of operators and the related Poisson pencil.} semisimple Nijenhuis $(1,1)$-tensor on $G/K$ with the real spectrum $\{\lambda_1,\dots,\lambda_s\}$, i.e., $($cf.\ Theorem~{\rm \ref{T1})} there exists a decomposition
\begin{gather}\label{dec}
\mathfrak{g} = \mathfrak{g}_1 + \cdots+ \mathfrak{g}_s
\end{gather}
to the sum of subspaces such that
\begin{itemize}\itemsep=0pt
\item $ \forall_{i,j\in \{1,\dots,s\}, i\not=j }$ $\mathfrak{g}_i\cap\mathfrak{g}_j = \mathfrak{k}$;
\item $ \forall_{i,j\in \{1,\dots,s\} }$ $\mathfrak{g}_i+\mathfrak{g}_j $ are Lie subalgebras in $\mathfrak{g}$;
\item the decomposition above induces the decomposition $T(G/K)=D_1\oplus\cdots\oplus D_s$ to integrable subbundles and $N|_{D_i}=\lambda_i\operatorname{Id}_{D_i}$.
\end{itemize}
Let $(\Pi,\Pi_1)$ be the Poisson pair consisting of the canonical Poisson bivector $\Pi$ on $T^*(G/K)$ and of the Poisson bivector $\Pi_1=\widetilde{N}\circ\Pi$, where $\widetilde{N}$ is the cotangent lift of the $(1,1)$-tensor $N$ $($see Definition~{\rm \ref{ctgl}} and Lemma~{\rm \ref{constt})}.

 Then the bi-Poisson structure generated by the reduced Poisson pair $(p_*\Pi,p_*\Pi_1)$ is Kronecker at any point of the set~$p(W)$, where $W \subset M_H$ is the open dense set which will be specified in the proof, if and only if for any $i=1,\dots,s$
 \begin{gather}\label{eq67}
\exists\, a_i\in (\mathfrak{g}/\check{\mathfrak{g}}_i)^*\colon \ \operatorname{ind} \mathfrak{g}^{a_i}+ \operatorname{codim}_{(\mathfrak{g}/\check{\mathfrak{g}}_i)^*}\mathcal{O}_{a_i} = \operatorname{ind} \mathfrak{g},
 \end{gather}
where $\check{\mathfrak{g}}_i=\sum\limits_{j\not=i}\mathfrak{g}_i$ and $\mathfrak{g}^{a_i}$ and $\mathcal{O}_{a_i}$ are respectively the stabilizer algebra and the orbit of the element $a_i$ with respect to the coadjoint action $ \operatorname{ad}^*\colon \check{\mathfrak{g}}_i \to \mathfrak{gl} (( {\mathfrak{g}}/{\check{\mathfrak{g}}_i })^*)$.

Equivalently, condition \eqref{eq67} can be written as
 \begin{gather}\label{eq78}
 \operatorname{ind} (\check{\mathfrak{g}}_i\ltimes (\mathfrak{g}/\check{\mathfrak{g}}_i))=\operatorname{ind} \mathfrak{g},
 \end{gather}
 where the term in the l.h.s.\ is the semidirect product of the Lie algebra $\check{\mathfrak{g}}_i$ and the vector space~$(\mathfrak{g}/\check{\mathfrak{g}}_i)$ with respect to the $\operatorname{ad}$-action.
 \end{Theorem}

\begin{proof} We first note that conditions (\ref{eq67}) and~(\ref{eq78}) are equivalent by Lemma~\ref{lrais} below.

Let $U=M_H\setminus\big(\bigcup_\lambda\mu_\lambda^{-1}(\operatorname{Sing}\mathfrak{g}^*)\big)$, where the moment map~$\mu_\lambda$ is specified in Lemma~\ref{lemk}(2).

Observe that all the objects involved admit a natural complexification (cf.\ Remark~\ref{rem}): the compact Lie groups $G$ and $K$ are imbedded in their Chevalley complexifications~$G^c$ and~$K^c$ and the homogeneous space $Q=G/K$ is imbedded into the complex homogeneous space \mbox{$Q^c=G^c/K^c$}. Moreover, the decomposition (\ref{dec}) implies the decomposition $\mathfrak{g}^\mathbb{C}=\mathfrak{g}_1^\mathbb{C}+\cdots+\mathfrak{g}_s^\mathbb{C}$, which in turn induces the decomposition $TQ^c=D_1^c\oplus\cdots\oplus D_s^c$ of the holomorphic tangent bundle to $Q^c$ to complex analytic involutive distributions and a complex analytic $(1,1)$-tensor~$N^c$ given by $N|_{D_i^c}=\lambda_i\operatorname{Id}_{D_i^c}$. By Lemma~\ref{lemk}(2) the assumptions of Theorem~\ref{below} are satisfied (it is well-known that for reductive Lie algebras $\operatorname{codim}\operatorname{Sing}\mathfrak{g}^*\ge 3$) and we conclude that the reduced bi-Poisson structure $\big\{\big(\Pi^\lambda\big)'\big\}$, $\big(\Pi^\lambda\big)'=p_*\Pi_1-\lambda p_*\Pi$, is Kronecker at a point $x'\in p(U)$ if and only if $\operatorname{corank} p_*\Pi^{\lambda_i}|_{x'}=\operatorname{ind} \mathfrak{g}$, $i=1,\dots,s$, where $\Pi^{\lambda_i}=\Pi_1-\lambda_i\Pi$ (see Theorem~\ref{below}(2)). Below we express the number $\operatorname{corank} p_*\Pi^{\lambda_i}|_{x'}$ in equivalent terms, see formula~(\ref{equiv}).

From Lemma \ref{lemk}(3) it follows that the restriction of the action $\tilde\rho\colon \mathfrak{g} \to \Gamma(T(T^*G/K))$ to the stabiliser subalgebra $\mathfrak{g}_{F} $ of any symplectic leaf $F$ of the Poisson bivector $\Pi^{\lambda_i}$ is Hamiltonian with respect to this bivector with the momentum map $\mu_{\lambda_i}^F\colon F\to \mathfrak{g}^*$. Obviously the action of $\mathfrak{g}_{F} $ is also locally free. Therefore by the bifurcation lemma (cf.\ the proof of Theorem~\ref{below}) the corank of the reduction $\big(\Pi^{\lambda_i}|F\big)'$ of the Poisson structure restricted to the symplectic leaf, $\Pi^{\lambda_i}|F$, at the point $x'=p(x)$, where $x\in F$, is equal to the index of the Lie algebra of $\mathfrak{g}_F$, provided $\mu_{\lambda_i}^F(x)\not\in\operatorname{Sing} \mathfrak{g}_F$. The algebraic set $\operatorname{Sing} \mathfrak{g}_F$ is nowhere dense in $\mathfrak{g}_F^*$ and the set $U_F=(F\cap U)\setminus\big(\bigcup_{i=1}^s\big(\mu^F_{\lambda_i}\big)^{-1}\big)(\operatorname{Sing} \mathfrak{g}_F)$ is an open dense set in~$F$ and, moreover, $V=\bigcup_FU_F$ is open and dense in~$M_H$. From now on we will consider only points $x'\in p(V)$.

Obviously, $\operatorname{corank}_{(M_H/G)}\big(\Pi^{\lambda_i}\big)_{x'}'= \operatorname{corank}_{F/G_F}\big(\Pi^{\lambda_i}|F\big)_{x'}'
+\operatorname{codim}_{\mathcal{S}_i}G \cdot F$.
Here $G_F$ is the subgroup in $G$ corresponding to the subalgebra $\mathfrak{g}_F$, $\mathcal{S}_i$ stands for the space of symplectic leaves of the Poisson bivector $\Pi^{\lambda_i}$, on which a natural action of the group $G$ is induced from the action of $G$ on $M$ due to the $G$-invariance of $\Pi^{\lambda_i}$, and $G \cdot F$ denotes the orbit of the point $F\in\mathcal{S}_i$ with respect to this action. Recall (see Lemma~\ref{L56}(1)) that the space $\mathcal{S}_i$ is foliated by the submanifolds of the form $\pi^{-1}(F_0)$, where $F_0 \subset G/K$ is a leaf of the foliation tangent to the distribution $\check D_i=\sum\limits_{j\not=i}D_j$.
Since the group $G$ acts transitively on $G/K$ and as a consequence on the space of leaves of the foliation tangent to the distribution~$\check D_i$, we have $\operatorname{codim}_{\mathcal{S}_i}G \cdot F=\operatorname{codim}_{\mathcal{S}_i|\pi^{-1}(\pi(F))}G_{\pi(F)} \cdot F$, where $G_{\pi(F)}$ is the subgroup corresponding to the subalgebra~$\mathfrak{g}_{\pi(F)}$, i.e., the stabilizer of the submanifold $\pi^{-1}(\pi(F))$ with respect to the cotangent action (see Lemma~\ref{lemk}(5)) and $\mathcal{S}_i|\pi^{-1}(\pi(F))$ stands for the submanifold in~$\mathcal{S}_i$ of leaves contained in the Poisson submanifold $\pi^{-1}(\pi(F))$. In view of Lemma~\ref{lemk}(7), Lemma \ref{stab}(4) and Lemma~\ref{lemk}(8) $\mathcal{S}_i|\pi^{-1}(\pi(F))$ can be identified with $(\mathfrak{g}/\mathfrak{g}_{\pi(F)})^*$, $G_{\pi(F)} \cdot F$ with $\mathcal{O}_{\varphi^i_F}$ and $\mathfrak{g}_F$ with $\mathfrak{g}^{\varphi^i_F}$, where $\varphi^i_F\in (\mathfrak{g}/\mathfrak{g}_{\pi(F)})^*$ is the functional corresponding to $F$ and $\mathfrak{g}^{\varphi^i_F}$ and $\mathcal{O}_{\varphi^i_F}$ are respectively its stabilizer and orbit with respect to the action of $\mathfrak{g}_{\pi(F)}$ on $(\mathfrak{g}/\mathfrak{g}_{\pi(F)})^*$.

Thus we have proven that $\operatorname{corank}_{(M_H/G)}\big(\Pi^{\lambda_i}\big)_{x'}'
=\operatorname{corank}_{F/G_F}\big(\Pi^{\lambda_i}|F\big)_{x'}'
+\operatorname{codim}_{\mathcal{S}_i}G \cdot F=\operatorname{ind} \mathfrak{g}_F
+\operatorname{codim}_{\mathcal{S}_i|\pi^{-1}(\pi(F))}G_{\pi(F)} \cdot F=\operatorname{ind} \mathfrak{g}^{\varphi^i_F}+\operatorname{codim}_{(\mathfrak{g}/\mathfrak{g}_{\pi(F)})^*}\mathcal{O}_{\varphi^i_F}$.
Finally, in view of Lem\-ma~\ref{stab}(2), we have $\mathfrak{g}_{\pi(F)}=\operatorname{Ad}_g\check {\mathfrak{g}}_i$ for some $g\in G$ and \begin{gather}\label{equiv}
 \operatorname{corank}\big(\Pi^{\lambda_i}\big)_{x'}'=
\operatorname{ind} \mathfrak{g}^{\varphi^i_F}+\operatorname{codim}_{(\mathfrak{g}/\operatorname{Ad}_g\check{\mathfrak{g}}_i)^*}\mathcal{O}_{\varphi^i_F}.
\end{gather}

We are ready to finish the proof. Assume that $\big\{\big(\Pi^\lambda\big)'\big\}$ is Kronecker at $x'$. Then by Theo\-rem~\ref{below} $\operatorname{ind} \mathfrak{g}_{\varphi^i_F}+\operatorname{codim}_{(\mathfrak{g}/\operatorname{Ad}_g\check{\mathfrak{g}}_i)^*}\mathcal{O}_{\varphi^i_F}=\operatorname{ind} \mathfrak{g}$, $i\in \{1,\dots,s\}$. Acting by $\operatorname{Ad}_{g^{-1}}$ we will get condition~(\ref{eq67}).

Vice versa, assume that (\ref{eq78}) is satisfied. Then by the Ra\"{i}s formula (see Lemma~\ref{lrais}) $\operatorname{ind} \mathfrak{g}= \operatorname{ind} \mathfrak{g}^{a_i}+ \operatorname{codim}_{(\mathfrak{g}/\check{\mathfrak{g}}_i)^*}\mathcal{O}_{a_i}$ for $a_i\in R((\mathfrak{g}/\check{\mathfrak{g}}_i)^*)$. Obviously also $\operatorname{ind} \mathfrak{g}= \operatorname{ind} \mathfrak{g}^{g\cdot a_i}+ \operatorname{codim}_{(\mathfrak{g}/\operatorname{Ad}_g\check{\mathfrak{g}}_i)^*}\mathcal{O}_{g\cdot a_i}$ and ${g\cdot a_i}\in R((\mathfrak{g}/\operatorname{Ad}_g\check{\mathfrak{g}}_i)^*)$ for any $g\in G$, where $g\cdot a_i:=\operatorname{Ad}^*_{g^{-1}}a_i$. Fix $g\in G$ and
let $F_i$ be the symplectic leaf of the Poisson bivector $\Pi^{\lambda_i}$ corresponding to the element ${g\cdot a_i}$ by Lemma~\ref{stab}(4) (with $\pi(F_i)=P\big(\operatorname{Ad}_g\check G_i\big)$). Note that the leaves $F_i$ are mutually transversal and $\sum_i\operatorname{codim} F_i=\dim M$, thus $\bigcap_i F_i$ is a point, say $x$.

Recall (see Lemma~\ref{lemk}(6),~(7) and Lemma~\ref{stab}(4)) that the map
\begin{gather*} \nu_i^g:=\nu_i|_{\pi^{-1}(\pi(F_i))}\colon \ \pi^{-1}(\pi(F_i))\to \mathfrak{g}_{\pi(F_i)}^\bot\cong (\mathfrak{g}/\operatorname{Ad}_g\check{\mathfrak{g}}_i)^*
\end{gather*} is an epimorphism. The set $\big(\nu_i^g\big)^{-1}(R((\mathfrak{g}/\operatorname{Ad}_g\check{\mathfrak{g}}_i)^*)$ is an open dense set in $\pi^{-1}(\pi(F_i))$ and the set $W=V\cap \big(\bigcup_{g\in G}\bigcap_{i=1}^s\big(\nu_i^g\big)^{-1}\big)(R((\mathfrak{g}/\operatorname{Ad}_g\check{\mathfrak{g}}_i)^*))$ is an open dense set in $M_H$.

Taking $a_i$ such that $g\cdot a_i\in \nu^g_i\big(W\cap \pi^{-1}(\pi(F_i))\big)$ for any $i$, we achieve that $x\in W \subset V$.
Formula~(\ref{equiv}) shows that $\operatorname{ind} \mathfrak{g}^{g\cdot a_i}+ \operatorname{codim}_{(\mathfrak{g}/\operatorname{Ad}_g\check{\mathfrak{g}}_i)^*}\mathcal{O}_{g\cdot a_i}=\operatorname{corank} p_*\Pi^{\lambda_i}|_{x'}$, where $x'=p(x)$. By Theorem~\ref{below} we conclude that $\{(\Pi^\lambda)'\}$ is Kronecker at~$x'$.
\end{proof}

\begin{Lemma}\label{lrais} Let $\mathfrak{g}$ be a Lie algebra and $\mathfrak{h} \subset\mathfrak{g}$ its Lie subalgebra. Then the condition of existing $a\in (\mathfrak{g}/\mathfrak{h})^*$ such that
\begin{gather}\label{rai1}
 \operatorname{ind} \mathfrak{h}^{a}+ \operatorname{codim}_{(\mathfrak{g}/\mathfrak{h})^*}\mathcal{O}_{a} = \operatorname{ind} \mathfrak{g},
 \end{gather}
where $\mathfrak{h}^{a}$ and $\mathcal{O}_{a}$ are respectively the stabilizer algebra and the orbit of the element $a$ with respect to the coadjoint action $ \operatorname{ad}^*\colon \mathfrak{h} \to \mathfrak{gl} (( {\mathfrak{g}}/{\mathfrak{h} })^*)$, is equivalent to the following one:
 \begin{gather}\label{rai2}
 \operatorname{ind} (\mathfrak{h}\ltimes (\mathfrak{g}/\mathfrak{h}))=\operatorname{ind} \mathfrak{g},
 \end{gather}
 where the Lie algebra in the l.h.s.\ is the semidirect product of the Lie algebra $\mathfrak{h}$ and the vector space~$(\mathfrak{g}/\mathfrak{h})$ with respect to the $\operatorname{ad}$-action.
 Moreover, if one of this condition holds, the equality~\eqref{rai1} holds for any $a$ from the open dense set $R((\mathfrak{g}/\mathfrak{h})^*) \subset (\mathfrak{g}/\mathfrak{h})^*$, $R((\mathfrak{g}/\mathfrak{h})^*):=\{\alpha\in (\mathfrak{g}/\mathfrak{h})^*\,|\, \exists\, \beta\in (\mathfrak{h}\ltimes (\mathfrak{g}/\mathfrak{h}))^*\setminus \operatorname{Sing} ((\mathfrak{h}\ltimes (\mathfrak{g}/\mathfrak{h}))^*)\colon \alpha= \beta|_{(\mathfrak{g}/\mathfrak{h})}\}$.
 \end{Lemma}

\begin{proof} Recall \cite[Proposition~1.3(i)]{rais} that
\begin{gather*}
\operatorname{ind} (\mathfrak{h}\ltimes (\mathfrak{g}/\mathfrak{h}))= \operatorname{ind} \mathfrak{h}^{a}+ \operatorname{codim}_{(\mathfrak{g}/\mathfrak{h})^*}\mathcal{O}_{a}
\end{gather*}
for $a\in R((\mathfrak{g}/\mathfrak{h})^*)$. Hence (\ref{rai2}) implies (\ref{rai1}). On the other hand, for arbitrary $a$ we have $ \operatorname{ind} \mathfrak{h}^{a}+ \operatorname{codim}_{(\mathfrak{g}/\mathfrak{h})^*}\mathcal{O}_{a}=\operatorname{corank} \Pi_{(\mathfrak{h}\ltimes (\mathfrak{g}/\mathfrak{h}))^*}|_\beta\ge \operatorname{ind} (\mathfrak{h}\ltimes (\mathfrak{g}/\mathfrak{h}))$, where $\beta\in \operatorname{Sing} (\mathfrak{h}\ltimes (\mathfrak{g}/\mathfrak{h}))^*$ is any element with $a= \beta|_{(\mathfrak{g}/\mathfrak{h})}$ (cf.~\cite[Theorem~1.1]{pRais}) and, moreover, the number $\operatorname{ind} (\mathfrak{h}\ltimes (\mathfrak{g}/\mathfrak{h}))$ is bounded below by $\operatorname{ind} \mathfrak{g}$ (since $\mathfrak{h}\ltimes (\mathfrak{g}/\mathfrak{h})$ is a contraction of $\mathfrak{g}$). Thus, if $\operatorname{ind} \mathfrak{h}^{a}+ \operatorname{codim}_{(\mathfrak{g}/\mathfrak{h})^*}\mathcal{O}_{a}=\operatorname{ind} \mathfrak{g}$ for some $a$, then $\operatorname{ind} (\mathfrak{h}\ltimes (\mathfrak{g}/\mathfrak{h}))= \operatorname{ind} \mathfrak{g}$.
\end{proof}

\begin{Remark}\label{remarr} In the case when $K=\{e\}$ is the trivial subgroup of the Lie group $G$, condition~(\ref{eq67}) coincides with the necessary and sufficient condition of kroneckerity of the Lie--Poisson pencil related to an algebraic Nijenhuis operator obtained in \cite[Theorem~2.5]{p5}.
\end{Remark}

\begin{Theorem}\label{thgeod} Retain the assumptions of Theorem~{\rm \ref{kro}} and assume that one of the equivalent conditions \eqref{eq67}, \eqref{eq78} hold. Let $b$ be the $G$-invariant metric on $G/K$, called {\em normal}, induced by some biinvariant metric on~$G$, i.e., by an $\operatorname{Ad} G$-invariant bilinear form $B$ on $\mathfrak{g}$. Then the $G$-invariant metric $b_N$, $b_N(x,y):=b((N+N^*)x,y)$, $x,y\in \Gamma(T(G/K))$, on $G/K$ corresponding to the symmetric $(1,1)$-tensor $N+N^*$, where $N^*$ is the adjoint to $N$ $(1,1)$-tensor, $b(N^*x,y)=b(x,Ny)$, as well as the normal metric itself have completely integrable geodesic flows in the class of analytic integrals polynomial in momenta.
\end{Theorem}

\begin{proof} It is well-known that a function of the form $\mu_{\rm can}^*f$, where $\mu_{\rm can}$ is the moment map of the $G$-action on $T^*(G/K)$ corresponding to the canonical Poisson bivector $\Pi$ and $f$ is any polynomial on~$\mathfrak{g}^*$, is analytic and polynomial in momenta. Indeed, the analyticity is obvious and the polynomiality can be argued as follows. If $\xi\in\mathfrak{g}$ is treated as a linear function on $\mathfrak{g}^*$, the function $H_\xi=\mu_{\rm can}^*\xi$ is the Hamiltonian function of the corresponding fundamental vector field~$V_\xi$, which in turn can be treated as a fiberwise linear function on~$T^*(G/K)$ (cf.\ Definition~\ref{defi2} and Remark~\ref{refund}). Thus, if~$f$ is a polynomial in~$\xi$, then $\mu_{\rm can}^*f$ is fiberwise polynomial.

By Theorem \ref{below}(3) the involutive set of functions $\mathcal{I}$ (\ref{setoffu}), where $\Pi_{t_0}=\Pi$, is complete on~$T^*(G/K)$. We have to prove that the quadratic forms $q(x):=b(x,x)$ and $q_N(x):=b_N(x,x)$, $x\in \Gamma(T^*(G/K))$, where we identified $T(G/K)$ with $T^*(G/K)$ by means of~$b$, is contained in this set. Let $Q(x)=B(x,x)$ be the quadratic form of $B$ understood as a Casimir function on~$\mathfrak{g}^*$ after the identification of~$\mathfrak{g}$ and~$\mathfrak{g}^*$ by means of~$B$. Then by Theorem~\ref{below}(4) the function~$\mu_{\rm can}^*Q$ belongs to~$\mathcal{I}$. One can show that in fact $\mu_{\rm can}^*Q$ coincides with~$q$. Indeed, $b$ belongs to the class of the so-called submersion metrics obtained from the right-invariant metrics on~$G$ by the canonical submersion $G\to G/K$. The quadratic forms of all the submersion metrics are of the form $\mu_{\rm can}^*f$, where $f$ is the corresponding quadratic polynomial on $\mathfrak{g}\cong\mathfrak{g}^*$ \cite[Section~7]{bjObzor}.

To prove that $q_N\in \mathcal{I}$ recall that by Theorem \ref{below}(4) the set $\mathcal{I}$ besides the functions $\mu_{\rm can}^*\mathcal{F}$ consists of the functions of the form $\mu_\lambda^*f$, $\lambda\not=\lambda_i$, $f\in Z^\Pi$. On the other hand, $\mu_\lambda=\mu_{\rm can}\circ((N-\lambda I)^{-1})^t$ by formula (\ref{mom}), hence the functions of the form $b\big(((N-\lambda I)^{-1})^*x,\big((N-\lambda I)^{-1}\big)^*x\big)=b((N-\lambda I)^{-1}\big((N-\lambda I)^{-1}\big)^*x,x)$, $x\in \Gamma(T^*(G/K))$, belong to $\mathcal{I}$. Moreover, $\mathcal{I}$ will contain also the coefficients of the Laurent expansion $b\big((N-\lambda I)^{-1}\big((N-\lambda I)^{-1}\big)^*x,x\big)=\frac1{\lambda^2}b(x,x)+\frac1{\lambda^3}b((N+N^*)x,x)+\cdots$ corresponding to the expansion $(N-\lambda I)^{-1}=-\big(\frac1{\lambda}I +\frac1{\lambda^2}N+\cdots\big)$.

Finally, the functions $\mu_\lambda^*f$, where $f$ are polynomial Casimir functions of $\Pi$, are polynomial in momenta (since $\big((N-\lambda I)^{-1}\big)^t$ is a fiberwise linear map).
\end{proof}

\section[Applications: two homogeneous spaces with integrable geodesic flows]{Applications: two homogeneous spaces\\ with integrable geodesic flows}\label{s.appl}

In the table from Example \ref{onish} among the triples $(\mathfrak{g},\mathfrak{g}_1,\mathfrak{g}_2)$ of compact Lie algebras such that $\mathfrak{g}=\mathfrak{g}_1+\mathfrak{g}_2$ of one can find two distinguished from our point of view series: $(A_{2n-1},C_n,A_{2n-2}\oplus T)$ and $(D_{n+1},B_n,A_n\oplus T)$. For both of them the pairs $(\mathfrak{g},\mathfrak{g}_i)$, $i=1,2$, are symmetric, i.e., the Lie algebra $\mathfrak{g}_i$ is the fixed point set of an automorphism of $\mathfrak{g}$ of second order (cf.\ \cite[Tables II, III, Section~6, Chapter~X]{helgason}). In this context we have to mentioned the following reformulation of the result of Brailov \cite[Theorem~5, Section~37]{TFbols}.

\begin{Theorem}\label{brailov} Let $\mathfrak{g}$ be a semisimple Lie algebra and $\mathfrak{g}_0 \subset\mathfrak{g}$ its symmetric subalgebra. Then \begin{gather*}
\operatorname{ind} \mathfrak{g}=\operatorname{ind} (\mathfrak{g}_0\ltimes(\mathfrak{g}/\mathfrak{g}_0)),
\end{gather*}
where $\mathfrak{g}_0\ltimes(\mathfrak{g}/\mathfrak{g}_0)$ is the so-called $\mathbb{Z}_2$-contraction of $\mathfrak{g}$, i.e., the semidirect product of the Lie algebra $\mathfrak{g}_0$ and the vector space $\mathfrak{g}/\mathfrak{g}_0$ with respect of the natural $\operatorname{ad}$-representation of~$\mathfrak{g}_0$ in $\mathfrak{g}/\mathfrak{g}_0$.
\end{Theorem}
In particular, it follows from this result that both the series of decompositions $\mathfrak{g}=\mathfrak{g}_1+\mathfrak{g}_2$ mentioned satisfy condition~(\ref{eq78}) of Theorem~\ref{kro}. This allows us to formulate the following theorem.

\begin{Theorem}\label{appl} Let $G/K$ be one of the following homogeneous spaces:
\begin{enumerate}\itemsep=0pt
\item[$(a)$] ${\rm SU}(2n)/({\rm S}({\rm U}(2n-1)\times {\rm U}(1))\cap {\rm Sp}(n))$;
\item[$(b)$] ${\rm SO}(2n+2)/({\rm SO}(2n+1)\cap {\rm U}(n+1))$.
\end{enumerate}
Then the geodesic flow of
\begin{enumerate}\itemsep=0pt
\item[$1)$] the normal metric $b$ on $G/K$ and
\item[$2)$] the $G$-invariant metric $b_N$ $($see Theorem~{\rm \ref{thgeod})} corresponding to the $G$-invariant Nijenhuis $(1,1)$-tensor~$N$ on $G/K$ with the real spectrum $\{\lambda_1,\lambda_2\}$, $\lambda_1\not=\lambda_2$, $\lambda_i\not=0$, related to the decomposition $\mathfrak{g}=\mathfrak{g}_1+\mathfrak{g}_2$ with $\mathfrak{k}=\mathfrak{g}_1\cap\mathfrak{g}_2$ by Theorem~{\rm \ref{T1}}
 \end{enumerate}
is completely integrable in the class of analytic integrals polynomial in momenta.

Here $\mathfrak{g}$ and $\mathfrak{k}$ are the Lie algebras of $G$ and $K$ respectively and the triples of subalgebras $(\mathfrak{g},\mathfrak{g}_1,\mathfrak{g}_2)$ are equal to $(A_{2n-1},C_n,A_{2n-2}\oplus T)$ and $(D_{n+1},B_n,A_n\oplus T)$ respectively. The explicit formulae for the embeddings $\mathfrak{g}_i\subset \mathfrak{g}$ as well as the decomposition $\mathfrak{k}^\perp=\mathfrak{k}_1\oplus\mathfrak{k}_2$ of the complementary to~$\mathfrak{k}$ space corresponding to the decomposition $\mathfrak{g}=\mathfrak{g}_1+\mathfrak{g}_2$ and the ``inertia operator'' $n_{\mathfrak{k}^\perp}+n_{\mathfrak{k}^\perp}^*=(N+N^*)|_{T_o(G/K)\cong \mathfrak{k}^\perp}$ $($here $o=P(e))$ are listed in Appendix~{\rm \ref{appe}}.
\end{Theorem}

\begin{proof} In view of Theorem \ref{brailov} the result will follow form Theorem~\ref{thgeod} if we ensure that the action of~$G$ on~$T^*(G/K)$ is locally free (which is an essential assumption of Theorem~\ref{thgeod}). Below we prove this fact, which is equivalent to the fact that the stabilizer $\mathfrak{k}^E:=\operatorname{stab}_\rho(E)$ of a generic element~$E$ in $\mathfrak{k}^{\perp}$ under the isotropy action $\rho\colon \mathfrak{k} \rightarrow \mathfrak{gl}\big(\mathfrak{k}^{\perp}\big)$ vanishes; here $\mathfrak{k}^{\perp}$ is the orthogonal complement to $\mathfrak{k}$ with respect to the (nondegenerate) Killing form on $\mathfrak{g}$ and we identify isotropy and coisotropy action by means of this form restricted to $\mathfrak{k}^\perp$. In other words, $\mathfrak{k}^E$ coincides with~$Z_\mathfrak{k}(E)$, the centralizer of the element $E\in\mathfrak{k}^{\perp}$ in $\mathfrak{k}$. In fact, since the function $\dim\mathfrak{k}^E$ is lower semicontinuous, it is enough to show the existence of an element $E$ with $\mathfrak{k}^E=\{0\}$. Note that it is sufficient to show the existence of such an element for the complexified action which we do below. We list explicit realizations of the complexifications $\big(\mathfrak{g}^\mathbb{C},\mathfrak{g}^\mathbb{C}_1,\mathfrak{g}^\mathbb{C}_2\big)$ for the above mentioned triples $(\mathfrak{g},\mathfrak{g}_1,\mathfrak{g}_2)$ as well as the subspace $\big(\mathfrak{k}^\mathbb{C}\big)^\perp$ complementary to the subspace $\mathfrak{k}^\mathbb{C}=\mathfrak{g}_1^\mathbb{C}\cap\mathfrak{g}_2^\mathbb{C}$ with respect to the Killing form. Besides, we indicate the element $E\in\big(\mathfrak{k}^\mathbb{C}\big)^{\perp}$ with~$\operatorname{stab}_\rho(E)=\{0\}$ and outline the proof of the last equality.
We consider separately cases~(a) and~(b).\footnote{We switch to modern notations and denote the classical Lie algebras by small Gothic letters.}

 \emph{Case $(a)$: $(\mathfrak{a}_{2n-1},\mathfrak{c}_n,\mathfrak{a}_{2n-2}\oplus \mathfrak{t})$.} 
\begin{gather*}\mathfrak{g}^\mathbb{C}=\mathfrak{sl}(2n,\mathbb{C}),\qquad
\mathfrak{g}_1^\mathbb{C} = \left\{ \left[ \begin{matrix} Z_{1}&Z_{2}\\ \noalign{\medskip}Z_{3}&-{Z_{1}}^{T}\end{matrix} \right] \,|\, Z_2=Z_2^T, \, Z_3=Z_3^T \right\} \cong \mathfrak{sp}(n,\mathbb{C}),
\\
\mathfrak{g}_2^\mathbb{C}=\left\{ \left[
\begin{matrix}
Z & \mathbf{0}\\
\mathbf{0}^T & -t
\end{matrix} \right] \,|\, Z \in \mathfrak{gl}(2n-1,\mathbb{C}), t=\operatorname{Tr} Z \right\}\cong \mathfrak{gl}(2n-1,\mathbb{C}) ,
\\
\mathfrak{k}^\mathbb{C}=\left\{
\left[ \begin{matrix}
\tilde{A} & \mathbf{0} & \tilde{B} & \mathbf{0} \\
\mathbf{0}^T & t & \mathbf{0}^T & 0 \\
\tilde{C} & \mathbf{0} & -\tilde{A}^T & \mathbf{0} \\
\mathbf{0}^T & 0 & \mathbf{0}^T & -t
\end{matrix}
\right] \,|\, \tilde{A},\tilde{B},\tilde{C} \in \mathfrak{gl}(n-1,\mathbb{C}), \, \tilde{B}=\tilde{B}^T, \, \tilde{C}=\tilde{C}^T,\, t\in\mathbb{C} \right\} \\
\hphantom{\mathfrak{k}^\mathbb{C}}{} \cong \mathfrak{sp}(n-1,\mathbb{C})\oplus \mathfrak{t},
\\
\big(\mathfrak{k}^\mathbb{C}\big)^\perp=
\left\{
\left[
\begin{array}{c|c|c|c}
A & \mathbf{v}_1 & B & \mathbf{v}_2 \\ \hline
\mathbf{u}_1^T & a & \mathbf{u}_2^T & b\tsep{2pt} \\ \hline
C & \mathbf{v}_3 & A^T & \mathbf{v}_4 \tsep{2pt}\\ \hline
\mathbf{u}_3^T & c & \mathbf{u}_4^T & a\tsep{2pt}
\end{array}
\right] \,|\, \mathbf{v}_i,\mathbf{u}_i \in \mathbb{C}^{n-1},\, B=-B^T, \,C=-C^T,\right.\\
\left. \hphantom{\big(\mathfrak{k}^\mathbb{C}\big)^\perp=}{}a=-\operatorname{Tr} A,\, b, c \in \mathbb{C}
\vphantom{\begin{array}{c|c|c|c}
A & \mathbf{v}_1 & B & \mathbf{v}_2 \\ \hline
\mathbf{u}_1^T & a & \mathbf{u}_2^T & b\tsep{2pt} \\ \hline
C & \mathbf{v}_3 & A^T & \mathbf{v}_4 \tsep{2pt}\\ \hline
\mathbf{u}_3^T & c & \mathbf{u}_4^T & a\tsep{2pt}
\end{array}}
\right\}.
\end{gather*}

The isotropy action $\rho\colon \mathfrak{k} \to \mathfrak{gl}\big( \mathfrak{k}^{\perp}\big)$ can be decomposed into direct sum of two invariant subspaces
\begin{gather*}
V_1:=\left\{
\left[
\begin{array}{c|c|c|c}
A & \mathbf{0} & B & \mathbf{0} \\ \hline
\mathbf{0}^T & 0 & \mathbf{0}^T & b\tsep{2pt} \\ \hline
C & \mathbf{0} & A^T & \mathbf{0} \tsep{2pt}\\ \hline
\mathbf{0}^T & c & \mathbf{0}^T & 0\tsep{2pt}
\end{array}
\right] \right\}, \qquad V_2:=\left\{
\left[
\begin{array}{c|c|c|c}
0 & \mathbf{v}_1 & 0 & \mathbf{v}_2 \\ \hline
\mathbf{u}_1 & 0 & \mathbf{u}_2 & 0 \\ \hline
0 & \mathbf{v}_3 & 0 & \mathbf{v}_4 \\ \hline
\mathbf{u}_3 & 0 & \mathbf{u}_4 & 0
\end{array}
\right]\right\},\end{gather*} and a trivial 1-dimensional representation which will be neglected.

Let $\rho\colon \mathfrak{k} \rightarrow \mathfrak{gl}(V_1 \oplus V_2 )$ be the coisotropy representation with the invariant subspaces $V_i$ and let $E_i\in V_i$.
Then obviously $\operatorname{stab}_{\rho}(E_1+E_2)= \operatorname{stab}_{\tilde{\rho}}(E_2)$, where $\tilde{\rho}:=\rho|_{\operatorname{stab}_{\rho}(E_1)}$.

Take the element
\begin{gather*} \mathfrak{k}^\perp \ni E_1 =
\left[
\begin{array}{c|c|c|c}
A & 0 & 0 & 0 \\ \hline
0 & 0 & 0 & 1 \\ \hline
0 & 0 & A^T & 0 \tsep{2pt}\\ \hline
0 & 0 & 0 & 0
\end{array}
\right],\end{gather*} with $A= \left[ \begin{smallmatrix}
0		&	1		&	 			&					\\
&	\ddots	& 	\ddots		&					 \\
&			& 				&			1	 \\
&			&				&			0	 \\
\end{smallmatrix} \right] $, the standard nilpotent matrix.
Then $\operatorname{stab}_{\rho}(E_1)$ consists of the matrices of the form
\begin{gather*}
Y=\left[ \begin{array}{c|c|c|c}
A & 0 & B & 0 \\ \hline
0 & 0 & 0 & 0 \\ \hline
C & 0 & -A^T & 0 \tsep{2pt}\\ \hline
0 & 0 & 0 & 0
\end{array} \right],\end{gather*}
 where
 \begin{gather*} A=\left[\begin{matrix}
a_1 & a_2 & \dots & a_{n-1} \\
& a_1 & \ddots & \vdots \\
& & \ddots & a_2 \\
& & & a_1
\end{matrix}\right],\qquad B=\left[ \begin{matrix}
b_1 & \dots & b_{n-2} & b_{n-1} \\
\vdots & \reflectbox{$\ddots$} & b_{n-1} & \\
b_{n-2} & \reflectbox{$\ddots$} & & \\
b_{n-1} & & &
\end{matrix} \right] ,\\ C=\left[ \begin{matrix}
& & & c_1 \\
& & c_1 & c_2 \\
& \reflectbox{$\ddots$} & \reflectbox{$\ddots$} & \vdots \\
c_1 & c_2 & \dots & c_{n-1}
\end{matrix} \right] .
\end{gather*}
Choose $ E_2\in V_2$
with
$\mathbf{u}_1=\mathbf{v}_1^T=(1,\dots,1)$ and trivial $\mathbf{v}_i$, $\mathbf{u}_i$, $i>1$. Then for $Y \in \operatorname{stab}_{\rho}(E_1)$ we have
\begin{gather*} [E_2,Y]=
\left[
\begin{array}{c|c|c|c}
0 & \mathbf{u} & 0 & 0 \\ \hline
\mathbf{v} & 0 & \mathbf{w}_2 & 0 \\ \hline
0 & \mathbf{w}_1 & 0 & 0 \\ \hline
0 & 0 & 0 & 0
\end{array}
\right],\end{gather*}
where
\begin{gather*} \mathbf{u}=\left[\begin{array}{l}
a_{1}+\dots +a_{n-2} +a_{n-1}\\
a_{1}+\dots +a_{n-2}\\
\vdots \\
a_{1}
\end{array}\right]
,\qquad
\mathbf{v}=-\left[\begin{array}{l}
a_{1}\\
a_{1}+a_{2}\\
\vdots \\
a_{1}+a_{2}+\dots +a_{n-1}
\end{array}\right]^T
,\\
\mathbf{w}_1=-\left[\begin{array}{r}
c_{1}+c_{2}+\dots+c_{n-1}\\
c_{2}+\dots+c_{n-1}\\
\vdots\\
c_{n-1}
\end{array}\right], \qquad
\mathbf{w}_2=\left[\begin{array}{r}
b_{1}+b_2+\dots+b_{n-1}\\
b_{2}+\dots+b_{n-1}\\
 \vdots\\
b_{n-1}
\end{array}\right]^T.
\end{gather*}
If $Y\in\operatorname{stab}_{\tilde\rho}(E_2)$, then $a_i=0$, $b_i=0$, and $c_i=0$, which implies $\operatorname{stab}_{\rho}(E) =\{0\} $, where $E:=E_1+E_2$.

\emph{Case $(b)$: $(\mathfrak{d}_{n+1},\mathfrak{b}_n,\mathfrak{a}_n\oplus \mathfrak{t})$.}
\begin{gather*}
 \mathfrak{g}^\mathbb{C}=\left\{ \left[ \begin{matrix} Z_{1}&Z_{2}\\ Z_{3}&-{Z_{1}}^{T}\end{matrix} \right] \,|\, Z_i\in \mathfrak{gl}(n+1,\mathbb{C}),\, Z_2=-Z_2^T,\, Z_3=-Z_3^T \right\} \cong \mathfrak{so}(2n+2,\mathbb{C}) ,\\
 \mathfrak{g}_1^\mathbb{C}=\left\{ \left[
\begin {array}{c|c|c|c}
0&\mathbf{u}^T&0&\mathbf{v}^T\\ \hline
-\mathbf{v}&W_{1}&-\mathbf{v}&W_{2}\\ \hline
0&\mathbf{u}^T&0&\mathbf{v}^T\tsep{2pt}\\ \hline
-\mathbf{u}&W_{3}&-\mathbf{u}&-W_{1}^{T}\tsep{2pt}
\end {array} \right] \!|\, W_2=-W^T_2,\, W_3=-W_3^T,\, \mathbf{v},\mathbf{u} \in \mathbb{C}^n\! \right\}\! \cong \mathfrak{so}(2n+1,\mathbb{C}),\\
\mathfrak{g}_2^\mathbb{C}=\left\{ \left[ \begin{matrix} \tilde{A}&0\\ 0&-\tilde{A}^{T}
\end{matrix} \right] \in \mathfrak{g} \,|\, \tilde{A}\in \mathfrak{gl}(n+1,\mathbb{C}) \right\} \cong \mathfrak{gl}(n+1,\mathbb{C}),\\
\mathfrak{k}^\mathbb{C}=\left\{ Y:=\left[
\begin {array}{c|c|c|c}
0&\mathbf{0}^T&0&\mathbf{0}^T \\ \hline
\mathbf{0}&A&\mathbf{0}&0 \\ \hline
0&\mathbf{0}^T&0&\mathbf{0}^T \tsep{2pt}\\ \hline
\mathbf{0}&0&\mathbf{0}&-{A}^{T}\tsep{2pt}
\end {array} \right] \,|\, A \in \mathfrak{gl}(n,\mathbb{C}) \right\}\cong \mathfrak{gl}(n,\mathbb{C}),\\
\big(\mathfrak{k}^\mathbb{C}\big)^\perp=\left\{Z:=\left[ \begin {array}{c|c}
\begin {array}{c|c} a&\mathbf{u}^T \\ \hline -\mathbf{v}& 0 \end {array} & Z_2 \\ \hline
Z_3 & \begin {array}{c|c} -a& \mathbf{v}^T\tsep{2pt} \\ \hline-\mathbf{u}&0 \end {array}
\end {array} \right] \,|\, Z_2=-Z_2^T,\, Z_3=-Z_3^T,\, \mathbf{u},\mathbf{v}\in \mathbb{C}^n,\, a\in \mathbb{C} \right\}.
\end{gather*}

For $Y\in\mathfrak{k}^\mathbb{C}$ and $Z\in\big(\mathfrak{k}^\mathbb{C}\big)^\perp$ as above with
\begin{gather*} Z_2=\left[ \begin {array}{c|c} 0&\mathbf{w}^T_2 \\ \hline -{\mathbf{w}}_2& B \end {array} \right] , \qquad Z_3=\left[ \begin {array}{c|c} 0&\mathbf{w}^T_3 \\ \hline -{\mathbf{w}}_3& C \end {array} \right],
\end{gather*}
where $B$ and $C$ are skew-symmetric $n\times n$ matrices, one has
\begin{gather*}
[Y,Z]=\left[
\begin{array}{c|c|c|c}
0&-\mathbf{u}^TA&0& (A\mathbf{w}_2)^T \\ \hline
A\mathbf{v}&0&-A\mathbf{w}_2&AB+BA^T \tsep{2pt}\\ \hline
0&-\mathbf{w}_3^TA&0&-(A\mathbf{v})^T \tsep{2pt}\\ \hline
A^T \mathbf{w}_3&-(A^TC+CA)&A^T \mathbf{u}&0\tsep{2pt}
\end{array} \right].
\end{gather*}
We will prove the triviality of the stabilizer of the element
$E=\left[ \begin{array}{c|c}
0 & J \\ \hline
J & 0
\end{array} \right]\in \mathfrak{k}^\perp $,
where $J=
\left[ \begin {smallmatrix}
0		& 1				& 			& \\
-1 		&				& 			& \\
& 				& \ddots	& 1 \\
& 				& -1		& 0
\end{smallmatrix} \right] $.
Observe that conditions $A\mathbf{w}_2=0,A^T \mathbf{w}_3=0$ for $\mathbf{w}_2=\mathbf{w}_3=(1,0,\dots,0)^T$ imply that $a_{1,i}=0$, $a_{i,1}=0$, where we put $A=||a_{ij}||_{i,j=1}^n$. Thus for any $Y\in \operatorname{stab}(E)$ the
matrix $A$ is of the form $\left[
\begin {array}{c|c} 0&0 \\ \hline 0& A_{n-1} \end {array} \right]$ where $ A_{n-1} \in\mathfrak{gl}(n-1,\mathbb{C})$.
Next, such $Y \in \operatorname{stab}(E)$ if and only if simultaneously
\begin{gather*} AB+BA^T=\left[ \begin{matrix}
0&0&0&0&0&0\\
0&0&a_{22}&a_{32}&\dots&a_{n,2} \\
0&-a_{22}&0&*&*&* \\
0&-a_{32}&*&0&*&*\\
0&\vphantom{\int^0}\smash[t]{\vdots}&*&*& &*\\
0&-a_{n,2}&*&*&*&0
\end{matrix}
\right]=0
\end{gather*} and
\begin{gather*} A^TC+CA=\left[ \begin{matrix}
0&0&0&0&0&0\\
0&0&a_{22}&a_{23}&\dots&a_{2,n} \\
0&-a_{22}&0&*&*&* \\
0&-a_{23}&*&0&*&*\\
0&\vphantom{\int^0}\smash[t]{\vdots}&*&*& &*\\
0&-a_{2,n}&*&*&*&0
\end {matrix}
\right]=0.
\end{gather*} Therefore $A$ has to be of the form $ \left[
\begin {array}{c|c} 0&0 \\ \hline 0& A_{n-2} \end {array} \right]$ with $ A_{n-2} \in\mathfrak{gl}(n-2,\mathbb{C})$. By induction we conclude that $A=0$, i.e., $\operatorname{stab}(E)=\{0\}$.
\end{proof}

\section{Concluding remarks}\label{concl}

We would like to note that in the proof of Theorem \ref{kro} we tried to maximally accurately indicate the open dense set $W$ such that the reduced bi-Poisson structure is Kronecker at any point of~$p(W)$ (and \emph{is not Kronecker} in the complement). This is important from the point of view of study qualitative analysis of the geodesic flow since outside this set the singularities of the corresponding lagrangian fibration appear (cf.~\cite{bolsIzoSing}).

The assumption of compactness of the Lie groups $G$ and $K$ which appeared in Theorem~\ref{kro} (see also Theorem~\ref{below}) was used in order to guarantee (1) the existence of complexification of the homogeneous space $G/K$ and as a consequence of other related objects; (2) the existence of an $G$-invariant open dense set $M_0$ in $M=T^*(G/K)$ (the set $M_H$) such that the orbit space $M_0/G$ is a smooth manifold. In fact, the assumption of compactness can be essentially weakened (since conditions (1) and (2) can be achieved for a wider class of Lie groups) preserving the conclusion of the theorem. We did not discuss these weaker assumptions as the main application (Theorem~\ref{appl}) is aimed in the class of compact homogeneous spaces.

The assumption that the action of $G$ on $T^*(G/K)$ is free, which is essential in Theorems~\ref{below} and~\ref{kro}, can be bypassed by a special reduction to smaller groups instead of~$G$ and~$K$, see~\cite{pMyk} and~\cite{p4}.

We would like to mention that a matter of further research is the study of the cases when the necessary and sufficient conditions~(\ref{eq78}) are not satisfied. In such cases the canonical commuting set of functions $\mathcal{B}$ related to the reduced bi-Hamiltonian structure \emph{is not complete}. However, based on the experience from the study~\cite{pSymmetries} of bi-Hamiltonian structures related with Lie pencils (hence in fact reductions of $(T^*(G/K),\Pi,\Pi_1)$ with trivial~$K$) one could expect additional symmetries in this case and, as a consequence, additional Noether integrals. One can ask for algebraic conditions sufficient for the completeness of the family $\mathcal{B}$ enlarged by these integrals.

Finally, it is worth mentioning that our theory related to the triples $(\mathfrak{g},\mathfrak{g}_1,\mathfrak{g}_2)$ (see Section~\ref{s.appl}) is very close to that appearing within the generalized chain method \cite{bj3, bjObzor}. Note however, that the assumption of maximality of rank of the symmetric space, which is essential in \cite[Theorem~8.6]{bjObzor}, is not satisfied for our symmetric pairs $(\mathfrak{g},\mathfrak{g}_i)$, i.e., the overlap between the theories mentioned is minimal (and requires further study).

\appendix

\section[Compact real forms of the triples $(\mathfrak{g},\mathfrak{g}_1,\mathfrak{g}_2)$ and inertia operators]{Compact real forms of the triples $\boldsymbol{(\mathfrak{g},\mathfrak{g}_1,\mathfrak{g}_2)}$\\ and inertia operators}\label{appe}

Below we list explicit realizations of the compact real forms for the triples $(\mathfrak{g},\mathfrak{g}_1,\mathfrak{g}_2)$ used in Theorem \ref{appl} as well as the decompositions of the subspace $\mathfrak{k}^\perp=\mathfrak{k}_1\oplus\mathfrak{k}_2$ complementary to the subspace $\mathfrak{k}$ with respect to the Killing form induced by the decompositions $\mathfrak{g}=\mathfrak{g}_1+\mathfrak{g}_2$, where $\mathfrak{k}_i=\mathfrak{g}_i\cap\mathfrak{k}^\perp$, and formulae for the ``inertia operators'' $n_{\mathfrak{k}^\perp}+n^*_{\mathfrak{k}^\perp}\colon \mathfrak{k}^\perp\to\mathfrak{k}^\perp$ induced by the operator $n_{\mathfrak{k}^\perp}\colon \mathfrak{k}^\perp\to\mathfrak{k}^\perp$, $n_{\mathfrak{k}^\perp}|_{\mathfrak{k}_i}=\lambda_i\operatorname{Id}_{\mathfrak{k}_i}$. We also note that in both cases below the inertia operators are positive definite under the restrictions
\begin{gather*}
0<\lambda_1,\qquad 0<\lambda_2,\qquad
{\frac {\lambda_{1} }{\big(\sqrt {2}+1\big)^2}}< \lambda_{2}<{\frac {\lambda_{1} }{\big(\sqrt {2}-1\big)^2}}.
\end{gather*}

 \emph{Case $(\mathfrak{a}_{2n-1},\mathfrak{c}_n,\mathfrak{a}_{2n-2}\oplus \mathfrak{t})$.}
\begin{gather*}
\mathfrak{g}=\mathfrak{su}(2n)=\big\{A \in \mathfrak{sl}(2n,\mathbb{C})\,|\, A=-\overline{A}^T\big\},\\
\mathfrak{g}_1'=\left\{ \left[
\begin{matrix} Z_{1}&Z_{2}\\ -\overline{Z}_{2}& \overline{Z}_1\end{matrix} \right] \,|\, Z_1=-\overline{Z}_1^T,\, Z_2=Z_2^T \right\}\cong\mathfrak{sp}(n),\\
\mathfrak{g}_2=\left\{ \left[ \begin{matrix} Z &\mathbf{0}\\ \mathbf{0}^T & -t \end{matrix} \right] \,|\, Z\in \mathfrak{u}(2n-1),\, t=\operatorname{Tr} Z \right\}\cong\mathfrak{su}(2n-1)\oplus\mathfrak{t}, \\
\mathfrak{k}=\left\{\left[ \begin{array}{c|c|c|c}
Z_1 & \mathbf{0} & Z_2 & \mathbf{0} \\ \hline
\mathbf{0}^T & {\rm i}t & \mathbf{0}^T & 0\tsep{2pt} \\ \hline
-\overline{Z}_2 & \mathbf{0} & \overline{Z}_1 & \mathbf{0}\tsep{2pt} \\ \hline
\mathbf{0}^T & 0 & \mathbf{0}^T & -{\rm i}t\tsep{2pt}
\end{array}
\right] \,|\, Z_i\in \mathfrak{gl}(n-1,\mathbb{C}), \, Z_1=-\overline{Z}_1^T, \, Z_2=Z_2^T, \, t\in\mathbb{R} \right\}\\
\hphantom{\mathfrak{k}}{} \cong \mathfrak{sp}(n-1)\oplus \mathfrak{t},\\
\mathfrak{k}^\perp=\left\{X:=\left[ \begin{array}{c|c|c|c}
Z_1 & \mathbf{v} & Z_2 & \mathbf{u} \\ \hline
-\overline{\mathbf{v}}^T & a & \mathbf{u_1}^T & z \tsep{2pt}\\ \hline
\overline{Z}_2 & -\overline{\mathbf{u}}_1 & -\overline{Z}_1& \mathbf{v}_1\tsep{2pt} \\ \hline
- \overline{\mathbf{u}}^T & -\overline{z} & -\overline{\mathbf{v}}_1^T & a\tsep{2pt}
\end{array}
\right]\,|\, Z_1=-\overline{Z}_1^T, \,Z_2=-Z_2^T, \, \mathbf{v},\mathbf{u},\mathbf{v}_1,\mathbf{u}_1 \in \mathbb{C}^{n-1},\right.\\
\left.\hphantom{\mathfrak{k}^\perp= }{}
\vphantom{\left[ \begin{array}{c|c|c|c}
Z_1 & \mathbf{v} & Z_2 & \mathbf{u} \\ \hline
-\overline{\mathbf{v}}^T & a & \mathbf{u_1}^T & z \tsep{2pt}\\ \hline
\overline{Z}_2 & -\overline{\mathbf{u}}_1 & -\overline{Z}_1& \mathbf{v}_1\tsep{2pt} \\ \hline
- \overline{\mathbf{u}}^T & -\overline{z} & -\overline{\mathbf{v}}_1^T & a\tsep{2pt}
\end{array}
\right]}{}
z\in \mathbb{C},\, a = - \operatorname{Tr} Z_1 \right\}
,\\
\mathfrak{k}_1=\left\{\left[\begin{array}{c|c|c|c}
0		& 	\mathbf{v} & 0	&\mathbf{u} \\ \hline
-\overline{\mathbf{v}}^T &	0 & \mathbf{u}^T & z	\tsep{2pt} \\ \hline
0		& -\overline{\mathbf{u}} & 0 & \overline{\mathbf{v}} \\ \hline
- \overline{\mathbf{u}}^T & -\overline{z} & -\mathbf{v}^T & 0\tsep{2pt}
\end{array}\right]\right\} ,\!\!\qquad
 \mathfrak{k}_2= \left\{\left[\begin{array}{c|c|c|c}
Z_1 & \mathbf{v}_1 & Z_2 & 0 \\ \hline
-\overline{\mathbf{v}}_1^T & a & \mathbf{u}_1^T & 0 \tsep{2pt}\\ \hline
\overline{Z}_2 & -\overline{\mathbf{u}}_1 & -\overline{Z}_1 & 0\tsep{2pt} \\ \hline
0 & 0 & 0 & a \\
\end{array} \right] |\, a = -\operatorname{Tr} Z_1 \right\},\\
\frac1{2}\big(n_{\mathfrak{k}^\perp}+n_{\mathfrak{k}^\perp}^*\big)X=
\left[ \begin{array}{c|c|c|c}
\lambda_2Z_1 & \begin{array}{c} \lambda_2\mathbf{v}\\ +\frac{(\lambda_1-\lambda_2)}{2}\overline{\mathbf{v}}_1\end{array} & \lambda_2Z_2 & \begin{array}{c} \lambda_1\mathbf{u} \\ +\frac{(\lambda_1-\lambda_2)}{2}{\mathbf{u}}_1 \end{array} \\ \hline
\begin{array}{c} -\lambda_2\overline{\mathbf{v}}^T \\ -\frac{(\lambda_1-\lambda_2)}{2}{\mathbf{v}}_1^T\end{array} & \lambda_2 a & \begin{array}{c} \lambda_2\mathbf{u_1}^T\\ +\frac{(\lambda_1-\lambda_2)}{2}\mathbf{u}^T\end{array} & \lambda_1z \\ \hline
\lambda_2\overline{Z}_2 & \begin{array}{c} -\lambda_2\overline{\mathbf{u}}_1 \\ -\frac{(\lambda_1-\lambda_2)}{2}\overline{\mathbf{u}}\end{array} & -\lambda_2\overline{Z}_1& \begin{array}{c} \lambda_1\mathbf{v}_1 \\ +\frac{(\lambda_1-\lambda_2)}{2}\overline{\mathbf{v}}\end{array} \\ \hline
\begin{array}{c} - \lambda_1\overline{\mathbf{u}}^T\\ -\frac{(\lambda_1-\lambda_2)}{2}\overline{\mathbf{u}}_1^T\end{array} & -\lambda_1\overline{z} & \begin{array}{c} -\lambda_1\overline{\mathbf{v}}_1^T \\ -\frac{(\lambda_1-\lambda_2)}{2}\mathbf{v}^T\end{array} & \lambda_2 a
\end{array}
\right].
\end{gather*}

 \emph{Case $(\mathfrak{d}_{n+1},\mathfrak{b}_n,\mathfrak{a}_n\oplus \mathfrak{t})$}, cf.\ \cite[solution to Exercise~B.3, Chapter~VI]{helgason}.
\begin{gather*} \mathfrak{g}= \left\{ \left[
\begin{matrix}
W & Z \\
\overline{Z} & \overline{W}
\end{matrix} \right] \,|\, Z,W \in \mathfrak{gl}(n+1,\mathbb{C}),\, W=-\overline{W}^T, \,Z=-Z^T \right\}\cong\mathfrak{so}(2n+2,\mathbb{R}),\\
 \mathfrak{g}_1=
\left\{ \left[
\begin {array}{c|c|c|c}
0	&\mathbf{u}^T 	& 0 &\overline{\mathbf{u}}^T\\ \hline
-\overline{\mathbf{u}}&W_{1}&-\overline{\mathbf{u}}& Z_1\\ \hline
0&\mathbf{u}^T& 0 &\overline{\mathbf{u}}^T\tsep{2pt} \\ \hline
-\mathbf{u}&\overline{Z}_{1}&-\mathbf{u}& \overline{W}_1\tsep{2pt} \end {array} \right]\,|\,
 Z_1,W_1 \in \mathfrak{gl}(n,\mathbb{C}), \, Z_1=-Z_1^T, \, W_1=-\overline{W}_1^T,\, \mathbf{u} \in \mathbb{C}^n
\right\}\\
\hphantom{\mathfrak{g}_1}{} \cong\mathfrak{so}(2n+1,\mathbb{R}),\\
 \mathfrak{g}_2 =
\left\{ \left[ \begin{matrix} A&0\\ 0&-A^{T}
\end{matrix} \right] \,|\, A\in \mathfrak{u}(n+1) \right\}\cong \mathfrak{u}(n+1) ,\\
 \mathfrak{k} = \left\{ \left[
\begin {array}{c|c|c|c}
0&\mathbf{0}^T&0&\mathbf{0}^T \\ \hline
\mathbf{0}&W_1&\mathbf{0}&0 \\ \hline
0&\mathbf{0}^T&0& \mathbf{0}^T \tsep{2pt} \\ \hline
\mathbf{0}&0&\mathbf{0}&\overline{W}_1\tsep{2pt}
\end {array} \right] \,|\, W_1 =-\overline{W}_1^T \right\}\cong \mathfrak{u}(n),\\
\mathfrak{k}^\perp=\left\{X:= \left[
\begin {array}{c|c|c|c}
{\rm i}a	&\mathbf{u}^T 	& 0 &\overline{\mathbf{v}}^T\\ \hline
-\overline{\mathbf{u}}& 0 &-\overline{\mathbf{v}}& Z_1\\ \hline
0&\mathbf{v}^T& -{\rm i}a &\overline{\mathbf{u}}^T \tsep{2pt} \\ \hline
-\mathbf{v}&\overline{Z}_{1}&-\mathbf{u} & 0 \tsep{2pt} \end {array} \right] \,|\, Z=-Z^T, \, a \in \mathbb{R}, \, \mathbf{u},\mathbf{v}\in \mathbb{C}^n
\right\} ,\\
 \mathfrak{k}_1=\left\{
\left[
\begin {array}{c|c|c|c}
0	&\mathbf{u}^T 	& 0 &\overline{\mathbf{u}}^T\\ \hline
-\overline{\mathbf{u}}& 0 &-\overline{\mathbf{u}}& Z_1\\ \hline
0&\mathbf{u}^T& 0 &\overline{\mathbf{u}}^T\tsep{2pt} \\ \hline
-\mathbf{u}&\overline{Z}_{1}&-\mathbf{u} & 0 \tsep{2pt} \end {array} \right] \right\},\qquad
 \mathfrak{k}_2=\left\{ \left[ \begin {array}{c|c}
\begin {array}{c|c} {\rm i}a&\mathbf{u}_1^T \\ \hline -\overline{\mathbf{u}}_1& 0 \end {array} & 0 \\ \hline
0 & \begin {array}{c|c} -{\rm i}a&\overline{\mathbf{u}}_1^T \tsep{2pt} \\ \hline -\mathbf{u}_1&0 \end {array}
\end {array} \right] \right\} ,\\
\frac{1}{2}\big(n_{\mathfrak{k}^\perp}+n_{\mathfrak{k}^\perp}^*\big)X=
\left[
\begin {array}{c|c|c|c}
\lambda_2 {\rm i}a	& \begin{array}{c} \lambda_1 \mathbf{u}^T \\ + \frac{(\lambda_1 -\lambda_2)}{2}\overline{\mathbf{v}}^T\end{array} & 0 &
\begin{array}{c} \lambda_1 \overline{\mathbf{v}}^T\\ + \frac{(\lambda_1 - \lambda_2)}{2} \mathbf{u}^T\end{array} \\ \hline
\begin{array}{c} -\lambda_1 \overline{\mathbf{u}} \\ - \frac{(\lambda_1 -\lambda_2)}{2}\mathbf{v}\end{array} & 0 & \begin{array}{c} -\lambda_1 \overline{\mathbf{v}} \\ - \frac{(\lambda_1 - \lambda_2)}{2} \mathbf{u}\end{array} & \lambda_1 Z_1\\ \hline
0 & \begin{array}{c} \lambda_1 \mathbf{v}^T \\ + \frac{(\lambda_1-\lambda_2)}{2} \overline{\mathbf{u}}^T \end{array} & -\lambda_2 {\rm i}a &
\begin{array}{c} \lambda_1 \overline{\mathbf{u}}^T \\ + \frac{(\lambda_1 - \lambda_2)}{2} \mathbf{v}^T\end{array} \\ \hline
\begin{array}{c} -\lambda_1 \mathbf{v} \\ - \frac{(\lambda_1-\lambda_2)}{2} \overline{\mathbf{u}}\end{array} & \lambda_1 \overline{Z}_{1}&
\begin{array}{c} -\lambda_1 \mathbf{u} \\ - \frac{(\lambda_1 -\lambda_2)}{2}\overline{\mathbf{v}}\end{array} & 0 \end{array} \right].
\end{gather*}

\subsection*{Acknowledgements}

We are very grateful to anonymous referees for useful remarks which allowed to essentially improve the quality of our paper in its final version.

\pdfbookmark[1]{References}{ref}
\LastPageEnding

\end{document}